\documentclass[12pt,leqno,draft]{article}
\usepackage{amsfonts}
\usepackage{amsmath, amsthm, amsfonts, amssymb, color}
\usepackage{mathrsfs}
\usepackage{color}
\newtheorem{thm}{Theorem}[section]

\newtheorem{lem}[thm]{Lemma}

\newtheorem{exa}[thm]{Example}

\theoremstyle{definition}
\newtheorem{defn}{Definition}[section]
\newcommand{\scr}[1]{\mathscr #1}
\definecolor{wco}{rgb}{0.5,0.2,0.3}

\numberwithin{equation}{section} \theoremstyle{remark}

\newcommand{\ua}{\uparrow}

\title{{\bf 
Log-Harnack Inequality and Bismut Formula for McKean-Vlasov SDEs with Singularities in all Variables}\footnote{Supported in
 part by  National Key R\&D Program of China (No. 2022YFA1006000, 2020YFA0712900) and NNSFC (12271398).} }
\author{
{\bf   Xing Huang $^{a)}$,  Feng-Yu Wang $^{a)}$  }\\
\footnotesize{ a)Center for Applied Mathematics, Tianjin
University, Tianjin 300072, China}\\
\footnotesize{  xinghuang@tju.edu.cn, wangfy@tju.edu.cn}\\
}
\begin{document}
\allowdisplaybreaks
\def\R{\mathbb R}  \def\ff{\frac} \def\ss{\sqrt} \def\B{\mathbf
B} \def\W{\mathbb W}
\def\N{\mathbb N} \def\kk{\kappa} \def\m{{\bf m}}
\def\ee{\varepsilon}\def\ddd{D^*}
\def\dd{\delta} \def\DD{\Delta} \def\vv{\varepsilon} \def\rr{\rho}
\def\<{\langle} \def\>{\rangle} \def\GG{\Gamma} \def\gg{\gamma}
  \def\nn{\nabla} \def\pp{\partial} \def\E{\mathbb E}
\def\d{\text{\rm{d}}} \def\bb{\beta} \def\aa{\alpha} \def\D{\scr D}
  \def\si{\sigma} \def\ess{\text{\rm{ess}}}
\def\beg{\begin} \def\beq{\begin{equation}}  \def\F{\scr F}
\def\Ric{\text{\rm{Ric}}} \def\Hess{\text{\rm{Hess}}}
\def\e{\text{\rm{e}}} \def\ua{\underline a} \def\OO{\Omega}  \def\oo{\omega}
 \def\tt{\tilde} \def\Ric{\text{\rm{Ric}}}
\def\cut{\text{\rm{cut}}} \def\P{\mathbb P} \def\ifn{I_n(f^{\bigotimes n})}
\def\C{\scr C}      \def\aaa{\mathbf{r}}     \def\r{r}
\def\gap{\text{\rm{gap}}} \def\prr{\pi_{{\bf m},\varrho}}  \def\r{\mathbf r}
\def\Z{\mathbb Z} \def\vrr{\varrho} \def\ll{\lambda}
\def\L{\scr L}\def\Tt{\tt} \def\TT{\tt}\def\II{\mathbb I}
\def\i{{\rm in}}\def\Sect{{\rm Sect}}  \def\H{\mathbb H}
\def\M{\scr M}\def\Q{\mathbb Q} \def\texto{\text{o}} \def\LL{\Lambda}
\def\Rank{{\rm Rank}} \def\B{\scr B} \def\i{{\rm i}} \def\HR{\hat{\R}^d}
\def\to{\rightarrow}\def\l{\ell}\def\iint{\int}
\def\EE{\scr E}\def\Cut{{\rm Cut}}
\def\A{\scr A} \def\Lip{{\rm Lip}}
\def\BB{\scr B}\def\Ent{{\rm Ent}}\def\L{\scr L}
\def\R{\mathbb R}  \def\ff{\frac} \def\ss{\sqrt} \def\B{\mathbf
B}
\def\N{\mathbb N} \def\kk{\kappa} \def\m{{\bf m}}
\def\dd{\delta} \def\DD{\Delta} \def\vv{\varepsilon} \def\rr{\rho}
\def\<{\langle} \def\>{\rangle} \def\GG{\Gamma} \def\gg{\gamma}
  \def\nn{\nabla} \def\pp{\partial} \def\E{\mathbb E}
\def\d{\text{\rm{d}}} \def\bb{\beta} \def\aa{\alpha} \def\D{\scr D}
  \def\si{\sigma} \def\ess{\text{\rm{ess}}}
\def\beg{\begin} \def\beq{\begin{equation}}  \def\F{\scr F}
\def\Ric{\text{\rm{Ric}}} \def\Hess{\text{\rm{Hess}}}
\def\e{\text{\rm{e}}} \def\ua{\underline a} \def\OO{\Omega}  \def\oo{\omega}
 \def\tt{\tilde} \def\Ric{\text{\rm{Ric}}}
\def\cut{\text{\rm{cut}}} \def\P{\mathbb P} \def\ifn{I_n(f^{\bigotimes n})}
\def\C{\scr C}      \def\aaa{\mathbf{r}}     \def\r{r}
\def\gap{\text{\rm{gap}}} \def\prr{\pi_{{\bf m},\varrho}}  \def\r{\mathbf r}
\def\Z{\mathbb Z} \def\vrr{\varrho} \def\ll{\lambda}
\def\L{\scr L}\def\Tt{\tt} \def\TT{\tt}\def\II{\mathbb I}
\def\i{{\rm in}}\def\Sect{{\rm Sect}}  \def\H{\mathbb H}
\def\M{\scr M}\def\Q{\mathbb Q} \def\texto{\text{o}} \def\LL{\Lambda}
\def\Rank{{\rm Rank}} \def\B{\scr B} \def\i{{\rm i}} \def\HR{\hat{\R}^d}
\def\to{\rightarrow}\def\l{\ell}\def\BB{\mathbb B}
\def\8{\infty}\def\I{1}\def\U{\scr U} \def\n{{\mathbf n}}\def\v{V}
\maketitle

\begin{abstract} The log-Harnack inequality and Bismut formula are established for McKean-Vlasov SDEs with singularities in all (time, space, distribution) variables,
where the drift satisfies an integrability condition in time-space, and the continuity in distribution may be weaker than Dini.
  The main results considerably improve  the  existing   ones   for  the  case where the drift is $L$-differentiable and  Lipschitz continuous in distribution with respect to the 2-Wasserstein distance. \end{abstract} \noindent
 \vskip 2cm

\section{Introduction}

Let $\scr P$ be the set of all probability measures on $\R^d$ equipped with the weak topology, and let $W_t$ be an $m$-dimensional Brownian motion on a complete filtration probability space $(\OO,\{\F_t\}_{t\ge 0},\F,\P)$. Consider
the following McKean-Vlasov SDE on $\R^d$:
\beq\label{E0} \d X_t= b_t(X_t, \L_{X_t})\d t+  \si_t(X_t)\d W_t,\ \ t\in [0,T],\end{equation}
where  $T>0$ is a fixed time, $\L_{X_t}$ is the distribution of $X_t$, and
 $$b: [0,T]\times\R^d\times\tilde{\scr P}\to\R^d,\ \ \si: [0,T]\times\R^d\to \R^d\otimes\R^m$$
are measurable for some non-empty subspace  $\tilde{\scr P}\subset \scr P$ equipped with a complete    distance $\tilde{\rho}$. Because of its wide applications,    this type SDE has been intensively investigated, see for instance \cite{BB, BBP, CD,CK, L, MV, SZ} and the survey \cite{HRW}.

In this paper, we study the regularity of \eqref{E0}  for distributions in
$$\scr P_k:=\big\{\mu\in \scr P: \|\mu\|_k:=\mu(|\cdot|^k)^{\ff 1 k}<\infty\big\},\ \ k\in (1,\infty).$$
Note that $\scr P_k$    is a Polish space under the   Wasserstein distance
$$\W_k(\mu,\nu)= \inf_{\pi\in \C(\mu,\nu)} \bigg(\int_{\R^d\times\R^d} |x-y|^k \pi(\d x,\d y)\bigg)^{\ff 1 {k}},\ \  $$ where $\C(\mu,\nu)$ is the set of all couplings of $\mu$ and $\nu$.
The SDE \eqref{E0} is called well-posed for distributions in $\scr P_k$, if for any initial value $X_0$ with $\L_{X_0}\in \scr P_k$ (respectively, any initial distribution $\gg\in \scr P_k$), it has a unique solution (respectively, a unique weak solution)
$X=(X_t)_{t\in [0,T]}$ such that $\L_{X_\cdot}:=(\L_{X_t})_{t\in [0,T]}\in C([0,T];\scr P_k).$ In this case,  for any $\gg\in \scr P_k$, let
  $P_t^*\gg=\L_{X_t^\gg}$ for the solution $X_t^\gg$ with $\L_{X_0^\gg}=\gg$.  We study the regularity of the map
$$\scr P_k\ni \gg\mapsto P_tf(\gg):= \E[f(X_t^\gg)]=\int_{\R^d} f\d\{P_t^*\gg\}$$ for $ t\in (0,T]$ and $f\in\B_b(\R^d), $
  where $\B_b(\R^d)$ is the space of bounded measurable functions on $\R^d$.

As powerful tools characterizing the regularity in distribution for stochastic systems, the dimension-free Harnack inequality due to \cite{W97},  the log-Harnack inequality  introduced in  \cite{W10},  and the Bismut (also called Bismut-Elworthy-Li)   formula  developed from \cite{Bismut,EL}, have been intensively investigated. See for instance the monograph  \cite{Wbook} for an account of related study on SPDEs.

In recent years, the log-Harnack inequality and Bismut type formula have also been established for McKean-Vlasov SDEs with coefficients regular in the distribution variable. Below we present a brief summary.

Write $b_t(x,\mu)= b_t^{(0)}(x)+ b_t^{(1)}(x,\mu)$. According to \cite{FYW2}, if $b^{(0)}$ satisfies some integrability condition on $(t,x)$, and  there exists a constant $K_b\geq 0$
such that
$$|b_t^{(1)}(x,\mu)-b_t^{(1)}(y,\nu)|\leq K_b(|x-y|+\W_2(\mu,\nu)),\ \ (x,\mu), (y,\nu)\in\R^d\times \scr P_2,t\in[0,T],$$
then there exists a constant $c>0$ such that  the log-Harnack inequality
$$P_t \log f(\tt \gg)\le \log P_t f(\gg)+ \ff{c}{t} \W_2(\gg,\tt\gg)^2,\ \ t\in (0,T], \ f\in \B_b^+(\R^d), \ \gg,\tt\gg\in \scr P_2$$
holds,   where   $\B_b^+(\R^d)$ is the space of positive elements in $\B_b(\R^d)$. This inequality is equivalent to the entropy-cost inequality
$$\Ent(P_t^*\gg|P_t^*\tt\gg)\le \ff{c}t \W_2(\gg,\tt\gg)^2,\ \ t\in (0,T], \gg,\tt\gg\in \scr P_2,$$ where $\Ent$ is the relative entropy, i.e. for any $\mu,\nu\in \scr P$,
$\Ent(\nu|\mu):=\infty$ if $\nu$ is not absolutely continuous with respect to $\mu$, while
$$\Ent(\nu|\mu):=\mu(\rr\log\rr)=\int_{\R^d}(\rr\log\rr)\d\mu,\ \ \text{if}\ \rr:= \ff{\d\nu}{\d\mu}\ \text{exists. }$$ See also
  \cite{HW, RW21,FYW1} for  log-Harnack inequalities with more regular $b^{(0)}$, and see \cite{Ren} for the dimension-free   Harnack inequality with power.

If furthermore   $b_t^{(1)}(x,\mu)$ is $L$-differentiable in $\mu\in \scr P_k$, the following Bismut type formula has been established  in \cite{FYW3} for the intrinsic derivative
 $D^I_\phi$ (see Definition \ref{ind} below):
$$D_\phi^I P_tf(\mu)= \E\big[f(X_t^\mu) M_t^{\mu,\phi}\big],\ \ t\in (0,T], f\in \B_b(\R^d), \mu\in \scr P_k, \phi\in L^k(\R^d\to\R^d;\mu),$$ where
 $M_t^{\mu,\phi}$ is an explicit martingale. See \cite{BRW20, BBP,HSW, RW} for earlier results  with  more regular  $b^{(0)}$.
See \cite{B}  for the case where  $\mu=\delta_x$ is  the Dirac measure at $x\in\R^d$, and see   \cite{Song,T} for a less   explicit Bismut formula involving in the inverse of the Malliavin matrix of the solution.

 We emphasize
  that  existing results   on log-Harnack inequality and Bismut formula for McKean-Vlasov SDEs only
  apply to the case with coefficients  regular in the distribution variable, i.e. either  $\W_2$-Lipschitz continuous or  $L$-differentiable. The reason is that   the  Zvonkin transform technique \cite{AZ}  used in these references only kills singularities in the time-spatial variables  $(t,x)$, but not the distribution   variable.

On the other hand, a derivative estimate has been presented in \cite{CR} for the heat kernel when the drift is of type $b_t(x,\mu(V))$, where $V$ is  a H\"{o}lder continuous function, and  $\mu(V):= \int_{\R^d} V\d\mu$. In this case, the drift is only Lipschitz continuous in distribution with respect to
 $$\W_\vv(\mu,\nu):=\sup\big\{ \big|\mu(f)-\nu(f)\big|: |f(x)-f(y)| \le |x-y|^\vv\big\} $$
for some $\vv\in(0,1)$ rather than     $\W_1$,   and hence also has certain singularity in the distribution variable.  This  result encourages us  to establish the log-Harnack inequality and Bismut formula for McKean-Vlasov SDEs
with coefficients singular in all time-spatial-distribution variables.

Indeed, we will establish the log-Harnack inequality and Bismut formula  for McKean-Vlasov SDEs with  stronger singularity  in the  distribution variable:  the drift is only Lipschitz continuous   with respect to
\begin{align*} \W_\alpha(\mu,\nu):= \sup\big\{ \big|\mu(f)-\nu(f)\big|: |f(x)-f(y)| \le \aa( |x-y|) \big\},\end{align*}
where $\aa$ is the square root of a Dini function, i.e. it belongs  to  class
\beg{align*}   \scr A:= \bigg\{\aa: &\ [0,\infty)\to  [0,\infty)\  \text{is\ increasing\ and\ concave,\ } \\
&\aa(0)=0,\ \aa(r)>0\ \text{for}\ r>0,\  \int_0^1 \ff{\aa(r)^2}{r}\d r\in (0,\infty)\bigg\}.\end{align*}
Noting that  $\int_0^1 \ff{\aa(r)^2}{r}\d r<\infty$ is the Dini condition for $\aa^2$,
   the continuity  in the distribution variable is even weaker than  Dini,  so that the existing study in the literature is considerably improved.

The log-Harnack inequality is established in Section 2, where a key step is to  derive the estimate   (Lemma  \ref{DLP} for $k=2$):
$$ \W_\alpha(P_t^\ast\gg,P_t^\ast\tt\gg)\le    c  \W_2(\tt\gg,\gg)\ff{ \aa(t^{\ff 1 2}) }{\ss t},\ \ \gg,\tt\gg\in \scr P_2, t\in(0,T]  $$  for some constant $c>0$.

The Bismut formula for the intrinsic derivative of $P_tf$ is presented in Section 3, for which  we develop new techniques  to control  the intrinsic derivative $D^I$ and the extrinsic derivative $D^E$ of the drift term in the distribution variable (Theorem \ref{TA2'}(1)):
$$  \|D^I P_t [D^E b_t(y,\nu)(\cdot)](\mu)\|_{L^{\frac{k}{k-1}}(\mu)} \le \ff{c\,\aa(t^{\ff 1 2})}{\ss t},\ \ t\in (0,T],\mu\in\scr P_k, y\in\R^d, \nu\in\scr P_k.$$

\section{Log-Harnack Inequality}

Since $\W_2$ is involved in the log-Harnack inequality, in this section we  mainly  consider \eqref{E0} for  $(\tt{\scr P},\tt\rr)=(\scr P_2,\W_2)$, but the drift may be not Lipschitz continuous in $\W_k$ for any $k> 0$.
We first state the concrete assumption and the main result on the log-Harnack inequality, then present a complete proof in a separate subsection.

\subsection{Assumption and main result}

 We will allow $b_t(x,\cdot)$ to be merely Lipschitz continuous in the sum of $\W_2$ and the Wasserstein distance induced by the square root of a Dini function.

Let $\aa\in\scr A$. Then it holds
\beq\label{AA}  \aa(s+t)\le \aa(s)+\aa(t),\ \ \aa(rt)\le r\aa(t),\ \ \\ \ s,t> 0, r\ge 1.\end{equation}
These inequalities   follow  from $\aa(0)=0$ and the decreasing monotonicity of $\aa'$ such that
$$\aa'(s+t)\le \aa'(s),\ \ \ff{\d}{\d t} \aa(rt)= r\aa'(rt)\le  r\aa'(t),\ \ s,t\ge 0, r\ge 1.$$
The second estimate in \eqref{AA} with $r=t^{-1}$ yields
\begin{align}\label{AA2} \aa(t)\ge \aa(1)t>0,\ \ \ t\in (0,1].\end{align}

 To measure the   singularity   in $(t,x)\in [0,T]\times \R^d$, we recall locally integrable functional spaces presented in \cite{XXZZ}.
For any $t>s\ge 0$ and $p,q\in (1,\infty]$, we write $f\in  \tilde{L}_p^q([s,t])$  if $f: [s,t]\times\R^d\to \R$ is measurable with
 $$\|f\|_{\tilde{L}_p^q([s,t])}:= \sup_{y\in\R^d}\bigg\{\int_{s}^t \bigg(\int_{B(y,1)}|f(r,x)|^p\d x\bigg)^{\ff q p} \d r\bigg\}^{\ff 1q}<\infty,$$ where $B(y,1):=\{x\in \R^d: |x-y|\le 1\}$ is the unit ball centered at the point $y$.
 When $s=0$, we simply denote
\begin{align*} \tt L_p^q(t)=\tt L_p^q([0,t]),\ \ \|f\|_{\tilde{L}_p^q(t)}=\|f\|_{\tilde{L}_p^q([0,t])}.\end{align*}
 We   take $(p,q)$ from the   space
 \begin{equation*} \scr K:=\Big\{(p,q)\in(2,\infty]^2:\  \ \ff{d}p+\ff 2 q<1\Big\},\end{equation*}
and make the following assumption where $\nn$ is the gradient in $x\in\R^d$.

\begin{enumerate}
\item[{\bf(A)}]    Let  $(\tilde{\scr P},\tilde{\rho})=(\scr P_k,\W_k)$ for some $k\in(1,\infty)$.    There exist     $K\in (0,\infty),  l\in \mathbb N, \aa\in \scr A$   and
$$1\le f_i\in \tt L_{p_i}^{q_i}(T),\ \ (p_i,q_i)\in    \scr K,\ \ 0\le i\le l$$   such that the following conditions hold.
\item[$(A_1)$] $(\si_t\si^*_t)(x)$ is invertible and $\sigma_t(x)$ is weakly differentiable in $x$ such that
 $$\|\si\si^*\|_\infty+\|(\si\si^*)^{-1}\|_\infty<\infty,\ \ |\nabla\sigma|\le \sum_{i=1}^l f_i,$$
$$\lim_{\vv\downarrow 0} \sup_{t\in [0,T], |x-x'|\le \vv} \|(\sigma_t \sigma_t^{\ast})(x)-(\sigma_t \sigma_t^{\ast})(x')\|= 0.$$
\item[$(A_2)$]   $b_t(x,\mu)=b_t^{(0)}(x)+b_t^{(1)}(x,\mu),$  where for any $ t\in[0,T], x,y\in\R^d, \mu,\nu\in \scr P_k,$
\beg{align*} &|b^{(0)}_t(x)|\le f_0(t,x),\ \   |b_t^{(1)}(0,\delta_0)|\le K,\\
&|b_t^{(1)}(x,\mu)-b_t^{(1)}(y,\nu)|\leq K\big\{|x-y|+ \W_\aa(\mu,\nu)+\W_k(\mu,\nu)\big\}.\end{align*}
\end{enumerate}

We first observe   that  {\bf (A)} implies the well-posedness of \eqref{E0} for distributions in $\scr P_k$. Let $[\cdot]_\aa$ be the $\aa$-continuity modulus defined by
$$[f]_\aa:=\sup_{x\ne y} \ff{|f(x)-f(y)|}{\aa(|x-y|)}.$$
Since $\aa(0)=0$ and $\aa$ is concave, there exists a constant $c>0$ such that
\begin{equation*}\sup_{[f]_\aa\le 1} |f(x)-f(0)|\le \aa(|x|)\leq \alpha(1)(1+|x|)\le c +c|x|^k,\ \ x\in\R^d.\end{equation*}  Thus,
\begin{align} \label{alw}\ff 1 {c} \W_\aa(\mu,\nu)\le \W_{k,var}(\mu,\nu):=\sup_{|f|\le 1+|\cdot|^k} |\mu(f)-\nu(f)|.
\end{align}
So, by \cite[Theorem 3.1(1)]{FYW2} for $D=\R^d$,  under assumption {\bf (A)}, \eqref{E0} is well-posed for distributions in $\scr P_k$, and for any $n\ge 1$ there exists a constant $c_n>0$ such that
\beq\label{ES00} \E\Big[\sup_{t\in [0,T]} |X_t|^n\Big|\F_0\Big]\le c_n (1+|X_0|^n).\end{equation}
Consequently,
 \beq\label{ES2} \sup_{t\in[0,T]}\|P_t^*\gg\|_{k}^k = \sup_{t\in[0,T]}  (P_t^*\gg)(|\cdot|^k) \leq\E\Big[\sup_{t\in[0,T]}|X_t^\gg|^k\Big]\le c_k (1+\|\gg\|_k^k).\end{equation}

\begin{thm}\label{LHI} Assume {\bf (A)} with $k=2$.
Then there exists a constant $c>0$ such that
 \beg{align}\label{shlog}   \Ent(P_t^\ast\gamma |P_t^\ast\tilde{\gamma})
 \leq  \ff c {t}\,\W_2(\gamma,\tilde{\gamma})^2,\ \ t\in (0,T], \gamma,\tilde{\gamma}\in \scr P_{2}.\end{align}
\end{thm}
 \begin{exa} Let $h:\R^d\times\R^d\to\R^d$ satisfy
$$|h(x_1,y_1)-h(x_2,y_2)|\leq K_h|x_1-x_2|+\alpha(|y_1-y_2|),\ \ x_1,x_2,y_1,y_2\in\R^d$$
for some $\alpha\in\scr A$ and $K_h\geq 0$. Then $b_t^{(1)}(x,\mu)=\int_{\R^d}h(x,y)\mu(\d y)$ satisfies $(A_2)$.
\end{exa}
\subsection{Proof of Theorem \ref{LHI}}

Although in Theorem \ref{LHI} we assume  {\bf (A)} for $k=2$, for later use we will also consider general $k\in (1,\infty)$.
For any $\gg\in \scr P_k$, consider  the   decoupled SDE of \eqref{E0}:
\begin{align}\label{decop}\d X^{x,\gamma}_t= b_t(X^{x,\gamma}_t, P_t^*\gg)\d t +\si_t(X^{x,\gamma}_t)\d W_t,\ \ X^{x,\gamma}_0=x.
\end{align}
By \cite[Theorem 3.1(1)]{FYW2} for $D=\R^d$, this SDE is well-posed and  \eqref{ES00} also holds for $X_t^{x,\gg}$ in place of $X_t$, i.e. for any $n\ge 1$ there exists a  constant $c_n(\gg)>0$ such that
\beq\label{ES00'} \E\Big[\sup_{t\in [0,T]} |X_t^{x,\gg}|^n]\le c_n (\gg)(1+|x|^n),\ \ x\in\R^d.\end{equation}
Let $P_t^\gg$ be the associated Markov semigroup, i.e.
$$P_t^\gg f(x):=\E[f(X_t^{x,\gg})],\ \ t\in [0,T], x\in \R^d, f\in \B_b(\R^d).$$
We  first present the following generalized H\"older inequality with a concave function $\aa$.

\beg{lem}\label{LN} Let $\aa: [0,\infty)\to [0,\infty)$ be concave. Then for any non-negative random variables $\xi$ and $\eta$,
\begin{align}\label{ale}\E[\aa(\xi)\eta]\le \|\eta\|_{L^p(\P)} \aa\Big(\|\xi\|_{L^{\ff{p}{p-1}} (\P)}\Big),\ \ p> 1.\end{align}
Consequently, for any random variable $\bar{\xi}$ on $\R^d$,     $f\in C(\R^d;\BB)$ for a Banach space $(\BB,\|\cdot\|_\BB)$ with $[f]_\aa<\infty$, and any real random variable $ \bar{\eta}$ with $\E[ \bar{\eta}]=0$,
\begin{align}\label{aas}\big\|\E[f(\bar{\xi})\bar{\eta}]\big\|_\BB\le [f]_\aa  \|\bar{\eta}\|_{L^p(\P)} \aa\Big(\|\bar{\xi}-z\|_{L^{\ff{p}{p-1}} (\P)}\Big),\ \ p> 1, z\in\R^d.
\end{align}
\end{lem}

\beg{proof}  Since the assertion holds trivially for $p=\infty$, we only prove   for $p<\infty$.  It suffices to prove for $\E[\eta^p] \in (0,\infty)$. Let $\Q:=\ff{ \eta}{\E[\eta]}\P.$ By Jensen's and H\"older's inequalities, and  using the second inequality in \eqref{AA}, we obtain
\beg{align*} &\E[\aa(\xi)\eta]= \E[\eta] \E_\Q [\aa(\xi)]\le  \E[\eta]   \aa(\E_\Q[\xi])\le  \E[\eta]   \aa\bigg(\ff{(\E[\eta^p])^{\ff 1 p}}{\E[\eta]}\big(\E[\xi^{\ff{p}{p-1}}]\big)^{\ff{p-1}p} \bigg)\\
&\le   \E[\eta]   \bigg\{ \ff{(\E[\eta^p])^{\ff 1 p}}{\E[\eta]}\ \aa\Big(\big(\E[\xi^{\ff{p}{p-1}}]\big)^{\ff{p-1}p} \Big)\bigg\}= \big(\E[\eta^p]\big)^{\ff 1 p}  \aa\Big(\big(\E[\xi^{\ff{p}{p-1}}]\big)^{\ff{p-1}p}\Big).\end{align*}
Then the  second inequality follows by noting that $\E[\bar{\eta}]=0$ implies
$$\big\|\E[f(\bar{\xi})\bar{\eta}]\big\|_\BB=\big\| \E[\{f(\bar{\xi})-f(z)\}\bar{\eta}]\big\|_\BB\le [f]_\aa \E[\aa(|\bar{\xi}-z|)|\bar{\eta}|].$$
Therefore, the proof is completed.
\end{proof}
To characterize  properties of \eqref{decop},
 consider the following PDE for $u: [0,T]\times\R^d\to \R^d$:
\beq\label{PDE}
\frac{\partial }{\partial t}u_t(x)+(\L_t^\gamma u_t)(x)+ b_t^{(0)}(x)=\lambda u_t(x),\ \ u_T=0,
\end{equation}
where $\ll>0$  is a constant, and
\beq\label{gen}
\L_t^\gamma  :=\frac{1}{2}\mathrm{tr} \big\{(\sigma_t\sigma_t^\ast)\nabla^2\big\}+ b_t(\cdot,P_t^*\gg)\cdot\nabla.
 \end{equation}
By \cite[Theorem 2.1]{YZ} and {\bf (A)}, for large enough constants $\lambda,c>0$ independent of $\gamma$, \eqref{PDE} has  a unique solution ${u}^{\lambda,\gamma}$ satisfying
\begin{align}\label{u0}
\|u^{\ll,\gg}\|_\infty+\|\nabla  {u}^{\lambda,\gamma}\|_{\infty}\leq \frac{1}{5},\ \ \|\nabla^2  {u}^{\lambda,\gamma}\|_{\tilde{L}_{p_0}^{q_0}(T)}\leq c.
\end{align}
So, for any $t\in [0,T]$,
\begin{align}\label{The}x\mapsto \Theta^{\lambda,\gamma}_t(x):=x+ {u}^{\lambda,\gamma}_t(x),\ \ x\in \R^d
 \end{align}is a homeomorphism on $\R^d$.

 Moreover, for any $\gg\in \scr P_k$, $t\in[0,T]$, consider
\begin{equation}\label{acd}
\d \theta^{\lambda,\gamma}_t(x)=b^{(1)}_t((\Theta^{\lambda,\gamma}_t)^{-1}(\theta^{\lambda,\gamma}_t(x)),P_t^\ast\gamma)\d t, \ \  \theta^{\lambda,\gamma}_0(x)=\Theta^{\lambda,\gamma}_0(x), x\in\R^{d},
\end{equation}
and let \begin{equation}\label{thv}\tilde{\theta}^{\lambda,\gamma}_t(x)=(\Theta^{\lambda,\gamma}_t)^{-1}( \theta^{\lambda,\gamma}_t(x)),\ \ t\in [0,T], x\in\R^{d}.
\end{equation}
Then we have
\begin{equation}\label{thv1}
\d \Theta^{\lambda,\gamma}_t(\tilde{\theta}^{\lambda,\gamma}_t(x))=b^{(1)}_t(\tilde{\theta}^{\lambda,\gamma}_t(x),P_t^\ast\gamma)\d t, \ \ t\in [0,T], \tilde{\theta}^{\lambda,\gamma}_0(x)=x\in\R^{d}.
\end{equation}
\begin{lem}\label{poa01} Let $\si$ and $b$ satisfy  {\bf (A)}. Then the following assertions hold.
\beg{enumerate}
\item[$(1)$]For any $p\geq 1$, there exists a constant $c_p>0$ such that
\begin{align}\label{soi}
\E[|X^{x,\gamma}_t-\tilde{\theta}^{\lambda,\gamma}_t(x)|^p]\leq c_pt^{\frac{p}{2}},\ \ t\in[0,T], x\in\R^d, \gg\in \scr P_k.
\end{align}

 \item[$(2)$] For any $\alpha\in\scr A$, there exists a constant $c>0$ such that the gradient estimate holds:
\beq\label{BS001}\beg{split}&|\nabla P^{\gamma}_tf |(x):= \limsup_{|y-x|\to 0}\frac{|P^{\gamma}_tf(y) -P^{\gamma}_tf(x)|}{|y-x|}\\
&\leq \frac{c\alpha(t^{\frac{1}{2}})}{\sqrt{t}},\ \ \ \ [f]_\alpha\leq 1, \ x\in\R^d,\gamma\in\scr P_k,t\in(0,T].\end{split}\end{equation}
\end{enumerate}
\end{lem}

\begin{proof} (1)  We will use Zvonkin's transform defined in \eqref{The}.
 By It\^o's formula (see \cite[Lemma 3.3]{YZ}), we derive
\beq\label{E-X}\begin{split}
\d \Theta^{\lambda,\gamma}_t(X^{x,\gamma}_t)&= \big\{\lambda   u^{\lambda,\gamma}_t(X^{x,\gamma}_t) + b_t^{(1)}(X^{x,\gamma}_t, P_t^*\gg)\big\}\d t +\big\{(\nabla\Theta_t^{\lambda,\gamma}) \sigma_t\big\}(X^{x,\gamma}_t)\,\d W_t.
\end{split}\end{equation}
 By {\bf (A)}, \eqref{u0}, there exists a constant $C>1$ such that
\beg{align*} &C^{-1} |X^{x,\gamma}_t-\tilde{\theta}^{\lambda,\gamma}_t(x)|\le  \big| \Theta^{\lambda,\gamma}_t(X^{x,\gamma}_t)-\Theta^{\lambda,\gamma}_t(\tilde{\theta}^{\lambda,\gamma}_t(x))\big|\le C  |X^{x,\gamma}_t-\tilde{\theta}^{\lambda,\gamma}_t(x)|,\\
& \big| b_t^{(1)}(X^{x,\gamma}_t,   P_t^*\gg)-b_t^{(1)}(\tilde{\theta}^{\lambda,\gamma}_t(x),  P_t^*\gg)\big|\le C  |X^{x,\gamma}_t-\tilde{\theta}^{\lambda,\gamma}_t(x)|,\\
&\big|\lambda   u^{\lambda,\gamma}_t(X^{x,\gamma}_t) \big|+\big\|\big\{(\nabla\Theta_t^{\lambda,\gamma}) \sigma_t\big\}(X^{x,\gamma}_t)\big\|\le C,\ \ (t,x,\gg)\in [0,T]\times\R^d\times \scr P_k.\end{align*}
This together with \eqref{thv1}, \eqref{E-X} and Gronwall's inequality implies \eqref{soi}.

 (2) For any measurable $f:\R^d\to\R$ with $[f]_\alpha\leq 1$, take
$$f_n:= [(-n)\lor f]\land n,\ \ n\ge 1.$$
By an approximation technique, it is sufficient to prove \eqref{BS001} for $f\in\scr B_b(\R^d)$ with $ [f]_\alpha\leq 1$.
According to \cite[Theorem 4.1]{YZ}, there exists a constant $c_0>0$ such that for any $\gg\in\scr P_k$, the log-Harnack inequality
\begin{align*}P_t^\gg \log f(x)\le \log P_t^\gg f(y)+ \ff{c_0}t |x-y|^2,\ \ x,y\in\R^d, t\in (0,T], f\in \B_b^+(\R^d)\end{align*}
holds, so that \cite[Proposition 1.3.8]{Wbook} implies
$$|\nabla P^{\gamma}_tf |\leq \frac{\sqrt{2c_0}}{\sqrt{t}}\{P_t^\gamma |f|^2\}^{\frac{1}{2}},\ \ f\in\scr B_b(\R^d), t\in(0,T], \gamma\in\scr P_k.$$
Observe that for any $f\in\scr B_b(\R^d)$ with $[f]_\alpha\leq 1$,
\begin{align}\label{grd}|\nabla P^{\gamma}_tf |(x)
\nonumber&\leq \inf_{z\in\R}\frac{\sqrt{2c_0}}{\sqrt{t}}\{P_t^\gamma (|f-z|^2)(x)\}^{\frac{1}{2}}\\
&\leq \frac{\sqrt{2c_0}}{\sqrt{t}}\{\E(\alpha(|X_t^{x,\gamma}-\tilde{\theta}^{\lambda,\gamma}_t(x)|)^2)\}^{\frac{1}{2}}\\
\nonumber&\leq \frac{\sqrt{2c_0}}{\sqrt{t}}\alpha(\{\E(|X_t^{x,\gamma}-\tilde{\theta}^{\lambda,\gamma}_t(x)|^2)\}^{\frac{1}{2}}),\ \ x\in\R^d, t\in(0,T],
\end{align}
where in the last step, we used \eqref{ale} for $\eta=\alpha(\xi)$ with $\xi=|X_t^{x,\gamma}-\tilde{\theta}^{\lambda,\gamma}_t(x)|$ and $p=2$.
 Therefore, \eqref{BS001} follows from \eqref{grd}, \eqref{soi} and \eqref{AA}.
\end{proof}

To verify  \eqref{shlog}, in the following Lemma \ref{L1} and Lemma \ref{DLP} we will prove
\begin{align}\label{deset}\int_0^t\{\W_\alpha(P_s^*\gg,P_s^*\tilde{\gg})+\W_2(P_s^*\gg,P_s^*\tilde{\gg})\}^2\d s\leq c\W_2(\gamma,\tilde{\gamma})^2,\ \ t\in[0,T],\gamma,\tilde{\gamma}\in\scr P_2
\end{align}
for some constant $c>0$.

\begin{lem}\label{L1}
Assume {\bf(A)}. Then there exists a constant $c>0$ such that
\begin{align}\label{WAr}
\W_k(P_t^\ast\gamma,P_t^\ast \tt\gamma)\leq c \W_k(\gamma,\tt\gamma)+c \int_0^t \W_\alpha(P_s^\ast\gamma,P_s^\ast \tt\gamma) \d s,\ \ t\in [0,T], \gg,\tt\gg\in \scr P_k.
\end{align}
\end{lem}

\begin{proof} We take  $\F_0$-measurable random variables $X_0^\gg,X_0^{\tt\gg}$  such that
\beq\label{FGG} \L_{X_0^\gg}=\gg,\ \ \L_{X_0^{\tt\gg}}=\tt\gg,\ \ \W_k(\gg,\tt\gg)^k=\E[|X_0^\gg-X_0^{\tt\gg}|^k].\end{equation}
Recall that $\Theta^{\lambda,\gamma}_t$ is defined in \eqref{The}. By \eqref{PDE}, \eqref{gen} and It\^o's formula,
  we derive
  \beq\label{E-X1}\begin{split}
\d \Theta^{\lambda,\gamma}_t(X^{\gamma}_t)&= \big\{\lambda   u^{\lambda,\gamma}_t(X^{\gamma}_t) + b_t^{(1)}(X^{\gamma}_t, P_t^*\gg)\big\}\d t +\big\{(\nabla\Theta_t^{\lambda,\gamma}) \sigma_t\big\}(X^{\gamma}_t)\,\d W_t,
\end{split}\end{equation}
and
\beq\label{EXv}\begin{split}
&\d \Theta^{\lambda,\gamma}_t(X^{\tt\gamma}_t) =\big\{ \lambda   u^{\lambda,\gamma}_t(X^{\tt\gamma}_t) + b_t^{(1)}(X^{\tt\gamma}_t, P_t^*\gg)\big\}\d t \\
&\qquad +\nabla\Theta^{\lambda,\gamma}_t(X^{\tilde{\gamma}}_t)[b_t(X^{\tt\gamma}_t,P_t^\ast\tt\gamma) -b_t(X^{\tt\gamma}_t,P_t^\ast\gamma)]\d t +\big\{(\nabla\Theta_t^{\lambda,\gamma})\sigma_t\big\}(X^{\tt\gamma}_t)\,\d W_t.
\end{split}\end{equation}
Combining this with  \eqref{E-X1} and  {\bf (A)},
we prove the  desired estimate  by  using the maximal functional inequality,  Khasminskii's estimate and stochastic Gronwall's inequality, see for instance the proof of \cite[Lemma 2.1]{HWJMAA}
for details. Below we simply outline the procedure.

By $(A_2)$ we have
\beg{align*} &|b_t(X^{\tt\gamma}_t,P_t^\ast\tt\gamma) -b_t(X^{\tt\gamma}_t,P_t^\ast\gamma)|+ |b_t^{(1)}(X^{\tt\gamma}_t,P_t^\ast\gamma) -b_t^{(1)}( X^{\gamma}_t,P_t^\ast\gamma)|\\
&\le K\big\{|X^{\gamma}_t-X^{\tt\gamma}_t|+\W_\aa(P_t^\ast\gamma, P_t^*\tt\gg)+ \W_k(P_t^*\gg, P_t^*\tt\gg)\big\}.\end{align*}
Combining this with \eqref{E-X1}, \eqref{EXv},   $(A_1)$,   the maximal functional inequality and Khasminskii's estimate  (see \cite[Lemma 2.1 and Lemma 4.1]{XXZZ}),    we derive
\beg{align*}&\d\big|\Theta^{\lambda,\gamma}_t(X^{\gamma}_t)-\Theta^{\lambda,\gamma}_t(X^{\tt\gamma}_t)\big|^{k+1}\le \d M_t+|X^{\gamma}_t-X^{\tt\gamma}_t|^{k+1}\d \scr L_t \\
&\qquad+  c_1\big\{\W_\aa(P_t^\ast\gamma, P_t^*\tt\gg)+ \W_k(P_t^*\gg, P_t^*\tt\gg)\big\}|\Theta^{\lambda,\gamma}_t(X^{\gamma}_t)-\Theta^{\lambda,\gamma}_t(X^{\tt\gamma}_t)\big|^{k}\d t,
 \end{align*}
where $c_1>0$ is a constant, $\scr L_t$ is an adapted increasing process with $\E[\e^{\delta \scr L_T}]<\infty$ for any $\delta>0$, and $M_t$ is a local martingale.
Since \eqref{u0} implies
$$\ff 1 2  |X^{\tt\gamma}_t-X^{\gamma}_t|\le  |\Theta^{\lambda,\gamma}_t(X^{\gamma}_t)-\Theta^{\lambda,\gamma}_t(X^{\tt\gamma}_t)\big|\le 2 |X^{\tt\gamma}_t-X^{\gamma}_t|,$$
by the  stochastic Gronwall  inequality   (see \cite[Lemma 3.7]{XZ}),  we find a constant $c_2>1$ such that
\beg{align*}& \bigg\{\E\Big[\sup_{s\in [0,t]}  |X^{\tt\gamma}_s-X^{\gamma}_s|^k\Big|\F_0\Big]\bigg\}^{1+k^{-1}}
- c_2 |X_0^\gg-X_0^{\tt\gg}|^{k+1} \\
  &\le  c_2   \int_0^t   \big\{\W_\aa(P_s^\ast\gamma, P_s^*\tt\gg)+ \W_k(P_s^*\gg, P_s^*\tt\gg)\big\}\E\Big[  |X^{\tt\gamma}_s-X^{\gamma}_s|^{k}\Big|\F_0\Big]\d s,\ \ t\in [0,T].\end{align*}
So, there exists a constant $c_3>0$ such that for any $t\in [0,T]$,
\beg{align*}& \E\Big[\sup_{s\in [0,t]}|X^{\tt\gamma}_s-X^{\gamma}_s|^k\Big|\F_0\Big] -  c_2   |X_0^\gg-X_0^{\tt\gg}|^k\\
 &\le   c_2   \bigg( \int_0^t   \big\{\W_\aa(P_s^\ast\gamma, P_s^*\tt\gg)+ \W_k(P_s^*\gg, P_s^*\tt\gg)\big\}\E\Big[  |X^{\tt\gamma}_s-X^{\gamma}_s|^{k}\Big|\F_0\Big] \d s\bigg)^{\ff k{k+1}}\\
 &\le \ff 1 2 \E\Big[\sup_{s\in [0,t]}|X^{\tt\gamma}_s-X^{\gamma}_s|^k\Big|\F_0\Big]+ c_3 \bigg( \int_0^t   \big\{\W_\aa(P_s^\ast\gamma, P_s^*\tt\gg)+ \W_k(P_s^*\gg, P_s^*\tt\gg)\big\}
 \d s\bigg)^{k}.\end{align*}
This together with \eqref{FGG} yields
\beg{align*} &  \W_k(P_t^*\gg,P_t^*\tt\gg)\le \sup_{s\in [0,t]} \big(\E[|X^{\tt\gamma}_s-X^{\gamma}_s|^k]\big)^{\ff 1 k}\\
&\le (2c_2)^{\ff 1 k} \W_k(\gg,\tt\gg) +  (2c_3)^{\ff 1 k}  \int_0^t   \big\{\W_\aa(P_s^\ast\gamma, P_s^*\tt\gg)+ \W_k(P_s^*\gg, P_s^*\tt\gg) \big\}\d s,\ \ t\in [0,T].\end{align*}
By Gronwall's inequality, this implies the desired estimate   for some constant $c>0$.
\end{proof}

Noting that $X_t^{x,\gg}$ solves \eqref{E0} if the initial value $x$ is random with distribution $\gg$, by the standard Markov property of $X_t^{x,\gg}$,  we have
\begin{align}\label{docfo}P_t f(\gamma):=\int_{\R^d} f(x)(P_t^*\gg)(\d x) =\int_{\R^d}P^{\gamma}_t f(x)\gamma(\d x),\ \ f\in\scr B_b(\R^d).
\end{align}
The following lemma provides a regularity estimate on $P_t^*$,   which together with Lemma \ref{L1} implies the desired \eqref{deset}.
\begin{lem}\label{DLP}
Assume {\bf (A)}. 
Then there exists a constant $c>0$ such that
 \beq\label{FST}
   \W_\alpha(P_t^\ast\gg,P_t^\ast\tt\gg)\le    c  \W_k(\tt\gg,\gg)\ff{ \aa(t^{\ff 1 2}) }{\ss t},\ \
     t\in (0,T],  \gg,\tt\gg\in \scr P_k.\end{equation}
 Consequently,  there exists a constant $c>0$ such that for any $\gg,\tt\gg\in \scr P_k,$
\beq\label{END}\sup_{t\in [0,T]}  \W_k(P_t^\ast\gg,P_t^\ast\tt\gg)\le c \W_k(\gg,\tt\gg).\end{equation}
\end{lem}
\begin{proof}    By Lemma \ref{L1},  \eqref{END} follows from \eqref{FST} and  the fact $\int_0^T\frac{\alpha(t^{\frac{1}{2}})}{\sqrt{t}}\d t<\infty$. So, we only need to prove \eqref{FST}.

Let $X_0^\gg$ and $X_0^{\tt\gg}$ be in \eqref{FGG}.
For any $\vv\in [0,2]$, let
$$ X_0^{\gg^\vv}:=X_0^\gg+\vv(X_0^{\tt\gg}-X_0^\gg), \ \ \ \gg^\vv:=\L_{X_0^{\gg^\vv}}, $$ and let $X_t^{\gg^\vv}$ solve \eqref{E0} with initial value $X_0^{\gg^\vv}.$
Then
\beq\label{GG2}  \gg^{\vv}(|\cdot|)\le 2\|\gg\|_k+2\|\tt\gg\|_k,\ \ \vv\in [0,2],\end{equation}
\beq\label{TT}  \W_k(\gg^{\vv},\gg^{\vv+r})^k\le \E\big[|X_0^{\gg^\vv}- X_0^{\gg^{\vv+r}}|^k\big] = r^k \W_k(\gg,\tt\gg)^k,\ \ \vv,r\in [0,1].\end{equation}
For any $\vv\ge 0$, consider the SDE
\beq\label{Evv} \d X_t^{x,\gg^\vv}= b_t(X_t^{x,\gg^\vv}, P_t^*\gg^\vv)\d t+\si_t(X_t^{x,\gg^\vv})\d W_t,\ \ X_0^{x,\gg^\vv}= x, t\in [0,T].\end{equation}
For any $r\in (0,1)$, let
\beg{align*}  \eta_t^{\vv,r}=[\sigma_t^\ast (\sigma_t\sigma_t^\ast)^{-1}] (X^{x,\gamma^\vv}_t)[b_t(X^{x,\gamma^\vv}_t, P_t^*\gg^{\vv+r})-b_t(X^{x,\gamma^\vv}_t, P_t^*\gg^\vv)],\ \ t\in [0,T].
\end{align*}
By {\bf (A)},    there exists a constant $c_1>0$ such that
\begin{align}\label{etk}\sup_{t\in[0,T] }|\eta^{\varepsilon,r}_t|\leq c_1\big\{\W_\alpha(P_t^\ast\gamma^\varepsilon, P_t^\ast\gamma^{\vv+r})+\W_k(P_t^\ast\gamma^\varepsilon, P_t^\ast\gamma^{\vv+r})\big\},\ \ r,\vv\in [0,1].
\end{align}
By Girsanov's theorem,
$$R^{\varepsilon,r}_t:=\exp\left\{\int_0^t\<\eta^{\varepsilon,r}_s,\d W_s\>-\frac{1}{2}\int_0^t|\eta^{\varepsilon,r}_s|^2\d s\right\}, \ \ t\in[0,T]$$ is a martingale, and
   $$W_t^{\vv,r}=  W_t- \int_0^t \eta_s^{\vv,r}\d s,\ \ \ t\in [0,T]$$
 is a Brownian motion under  the probability measure $\Q^{\vv,r}:=R^{\varepsilon,r}_T \P.$
 Rewrite \eqref{Evv} as
\begin{align*}\d X^{x,\gamma^\vv}_t= b_t(X^{x,\gamma^\vv}_t, P_t^*\gg^{\vv+r})\d t  +\si_t(X^{x,\gamma^\vv}_t)\d W^{\vv,r}_t,\ \ X^{x,\gamma^\vv}_0=x,\ \ t\in [0,T].
\end{align*}
By the weak uniqueness we obtain
 $$\L_{\{X^{x,\gamma^\vv}_t\}_{t\in[0,T]}|\Q^{\vv,r}}=\L_{\{X^{x,\gamma^{\varepsilon+r}}_t\}_{t\in[0,T]}},$$ where $\L_{\cdot|\Q^{\vv,r}}$ is the law under $\Q^{\vv,r}$,
so that
\begin{align*} P_t^{\gamma^{\varepsilon+r}}f(x)-P^{\gamma^\vv}_tf(x)=\mathbb{E} \left[f(X^{x,\gamma^\vv}_t)(R_t^{\varepsilon,r}-1)\right],\ \ f\in \scr B_b(\R^d),\varepsilon,r\in(0,1].
\end{align*}
Hence, by \eqref{docfo}, we have
\begin{align*}
&P_tf (\gamma^{\varepsilon+r})-P_tf(\gamma^\vv)= \gamma^{\varepsilon+r}(P_t^{\gamma^{\varepsilon+r}}f)-\gamma^\vv(P^{\gamma^\vv}_tf)\\
&= \gamma^{\varepsilon+r} (P_t^{\gamma^{\varepsilon+r}}f-P^{\gamma^\vv}_tf) +\gamma^{\varepsilon+r}(P_t^{\gamma^\vv}f)-\gamma^\vv (P^{\gamma^\vv}_tf)\\
&=\int_{\R^d}\mathbb{E} \left[f(X^{x,\gamma^\vv}_t)(R_t^{\varepsilon,r}-1)\right] \gamma^{\varepsilon+r}(\d x)+ \E \big[P_t^{\gg^\vv} f(X_0^{\gamma^{\vv+r}})- P_t^{\gg^\vv} f(X_0^{\gg^\vv})\big],
\end{align*}
so that
\beq\label{PPo}\beg{split}
&\W_\aa(P_t^*\gg^{\vv+r}, P_t^*\gg^\vv)^2=\sup_{[f]_\aa\le 1} \big|P_tf(\gamma^{\varepsilon+r})-P_tf(\gamma^\vv)\big|^2\le I_1+I_2,\\
&I_1:= 2\sup_{[f]_\aa\le 1} \bigg|\int_{\R^d}\mathbb{E} \left[f(X^{x,\gamma^\vv}_t)(R_t^{\varepsilon,r}-1)\right] \gamma^{\varepsilon+r}(\d x)\bigg|^2,\\
& I_2:=2\sup_{[f]_\aa\le 1}\bigg|\E \big[P_t^{\gg^\vv} f(X_0^{\gamma^{\vv+r}})- P_t^{\gg^\vv} f(X_0^{\gg^\vv})\big]\bigg|^2.\end{split}
\end{equation}
Below we estimate $I_1$ and $I_2$ respectively.

By \eqref{etk}, we obtain
\begin{align}\label{RTo}
 \nonumber&\E|R_t^{\varepsilon,r}-1|^2= \E\big[(R_t^{\varepsilon,r})^2-1\big]
 \leq \mathrm{esssup}_{\Omega}(\e^{\int_0^t|\eta_s^{\varepsilon,r}|^2\d s}-1)\\
 &\leq \mathrm{esssup}_{\Omega}\left(\e^{\int_0^t|\eta_s^{\varepsilon,r}|^2\d s}\int_0^t|\eta_s^{\varepsilon,r}|^2\d s\right)\\
\nonumber &\leq \psi(\vv,r)\int_0^t\big\{\W_\alpha(P_s^\ast\gamma^\vv,P_s^\ast\gamma^{\vv+r})^2+\W_k(P_s^\ast\gamma^\vv,P_s^\ast\gamma^{\vv+r})^2\big\}\d s,
\end{align}  where  for $c_2:= 2 c_1^2$,
\beq\label{bar0} \psi(\vv,r):= c_2 \e^{c_2 \int_0^T\{\W_\alpha(P_s^\ast\gamma^\vv,P_s^\ast\gamma^{\vv+r})^2+\W_k(P_s^\ast\gamma^\vv,P_s^\ast\gamma^{\vv+r})^2\}\d  s}.\end{equation}
By    \eqref{alw}  and \eqref{ES2},  we have
\beq\label{bar} \bar\psi :=\sup_{\vv,r\in [0,1]}\psi(\vv,\gg)<\infty. \end{equation}
Combining this with \eqref{AA},  \eqref{soi}, \eqref{RTo}   and \eqref{aas} with $z=\tilde{\theta}^{\lambda,\gamma^\varepsilon}_t(x)$, where $\tilde{\theta}^{\lambda,\gamma^\varepsilon}_t(x)$ is defined in \eqref{thv} with $\gamma^\varepsilon$ replacing $\gamma$,  we can find       constants $k_1, k_2>1$ such that
 \beg{align*} &\bigg(\int_{\R^d} \sup_{[f]_\aa\le 1}  \Big|\E \left[f(X^{x,\gamma^\vv}_t)(R_t^{\varepsilon,r}-1)\right]\Big| \gg^{\vv+r}(\d x)\bigg)^2 \\
&\le   \bigg(\int_{\R^d}\aa\big(k_1 t^{\ff 1 2}\big) \sup_x\big(\E[|R_t^{\varepsilon,r}-1|^2]\big)^{\ff 1 2} \gg^{\vv+r}(\d x)\bigg)^2\\
&\le   \aa\big(k_1t^{\ff 1 2}\big)^2 \sup_{x}  \E[|R_t^{\varepsilon,r}-1|^2] \\
&  \leq k_2\aa\big( t^{\ff 1 2}\big)^2 \psi(\vv,r) \int_0^t\big\{\W_\alpha(P_s^\ast\gamma^\vv,P_s^\ast\gamma^{\vv+r})^2+\W_k(P_s^\ast\gamma^\vv,P_s^\ast\gamma^{\vv+r})^2\big\}\d s,\ \ t\in [0,T].
\end{align*}
Combining this with  \eqref{AA},  \eqref{WAr}, \eqref{TT},  \eqref{GG2},
 and letting
\beg{align*}
 \GG_t(\vv,r):=  \W_\aa(P_t^* \gg^\vv, P_t^*\gg^{\vv+r})^2+\int_0^t\W_\aa(P_s^* \gg^\vv, P_s^*\gg^{\vv+r})^2\d s,\end{align*}
we find a constant $c_4>0$ such that
 \beq\label{I1} I_1 \le  c_4\aa\big(\sqrt{T}\big)^2  \psi(\vv,r) \bigg(r^2 \W_k(\gg,\tt\gg)^2+ \int_0^t \GG_s(\vv,r)\d s\bigg),\ \ t\in [0,T].\end{equation}
By \eqref{BS001}, we find a constant $c_5>0$ such that
 \begin{equation*} \beg{split}
\sup_{[f]_\aa\le 1}\big|\nabla P^{\gamma^\vv}_tf \big|(x)
 & \le \ff{c_5}{\ss t} \aa\big( t^{\ff 1 2}\big). \end{split}\end{equation*}
 Combining this with \eqref{AA},
 we find a constant $c_6>0$ such that
\begin{align} \label{I21}\nonumber I_2  &\le 2\sup_{[f]_\aa\le 1} \bigg(\E\bigg[|X_0^\gg-X_0^{\tt\gg}|\int_0^r |\nn P_t^{\gg^\vv} f(X_0^{\gg^{\vv+\theta}})|\d \theta \bigg]\bigg)^2\\
&\le \ff{c_6\aa\big(t^{\ff 1 2}\big)^2}t r^2(\E|X_0^\gg-X_0^{\tt\gg}|)^2 \\
\nonumber&\le \ff{c_6\aa\big(t^{\ff 1 2}\big)^2r^2}t
\big(  \E[|X_0^\gg-X_0^{\tt\gg}|^k]\big)^{\ff 2 k}.\end{align}
 Let
\begin{align}\label{tal}\tilde{\aa}(r):=\bigg(\int_0^r \ff{\aa(t)^2}{t}\d t\bigg)^{\ff 1 2},\ \ \ r\ge 0.
\end{align}
By \eqref{AA}, we find  some constant $c'>0$ such that
\beq\label{ASS} \beg{split} & \int_0^T \ff{  \aa(r t^{\ff 1 2})^2}t \d t =   2   \int_0^{ rT^{\ff 1 2} } \ff{\aa(s)^2}s\d s\le c' \tt\aa(r)^2 <\infty,\ \ r\ge 1.\end{split}\end{equation}
 So, \eqref{I21} together with \eqref{PPo} and \eqref{I1} yields that for some constant $c_7>0$,
 \beq\label{LI3} \beg{split} &\GG_t(\vv,r)\le c_7  r^2\W_k(\gg,\tt\gg)^2H_t(\vv,r)
+  c_7\psi(\vv,r) \int_0^t \GG_s(\vv,r)\d s,\\
&H_t(\vv,r):=\psi(\vv,r)+ \tilde{\aa}\big(1\big)^2+\frac{\aa\big(t^{\ff 1 2}\big)^2}{t} ,\ \ \   \vv,r\in [0,1], t\in [0,T].\end{split}\end{equation}
By Gronwall's inequality and \eqref{LI3},  for any $   \vv,r\in [0,1]$ we have
\beg{align*} &\W_\aa(P_t^*\gg^\vv, P_t^*\gg^{\vv+r})^2\le\GG_t(\vv,r) \\
&\le  c_7  r^2\W_k(\gg,\tt\gg)^2 \bigg\{H_t(\vv,r)
  +  c_7\psi(\vv,r)  \e^{c_7\psi(\vv,r)T}\int_0^t H_s(\vv,r)\d s\bigg\},\ \ t\in [0,T].\end{align*}
This together with \eqref{WAr}, \eqref{ASS} and \eqref{bar0}-\eqref{bar}  implies that $\psi(\vv,r)$ is bounded in $(\vv,r)\in [0,1]^2$ with $\psi(\vv,r)\to c_2$ as $r\to 0$, so that by the dominated convergence theorem
 we find a constant $c >0$ such that
 \beq\label{AB} \beg{split}& \limsup_{r\downarrow 0}\ff{ \W_\aa(P_t^*\gg^\vv, P_t^*\gg^{\vv+r})}r  \le  c  \W_k(\tt\gg,\gg)\bigg\{\ff{ \aa(t^{\ff 1 2}) }{\ss t}   +1\bigg\}.
\end{split} \end{equation}
 By the triangle inequality,
$$\big| \W_\aa(P_t^*\gg, P_t^*\gg^\vv)- \W_\aa(P_t^*\gg, P_t^*\gg^{\vv+r})\big|\le \W_\aa(P_t^*\gg^\vv, P_t^*\gg^{\vv+r}),  \ \    \vv,r\in [0,1],$$
so that \eqref{AB}   implies that
  $ \W_\aa(P_t^*\gg, P_t^*\gg^\vv)$ is Lipschitz continuous (hence a.e. differentiable)  in $\vv\in [0,1]$ for any $t\in (0,T]$,  and
  \beg{align*} & \Big|\ff{\d}{\d \vv} \W_\aa(P_t^*\gg, P_t^*\gg^{\vv})\Big|\le   \limsup_{r\downarrow 0}\ff{ \W_\aa(P_t^*\gg^\vv, P_t^*\gg^{\vv+r}) }r\le c  \W_k(\tt\gg,\gg)\bigg\{\ff{ \aa(t^{\ff 1 2}) }{\ss t}   +1\bigg\},\ \ \vv\in [0,1].\end{align*}
This implies \eqref{FST} by noting that $\gg^1=\tt\gg$  and $\sup_{t\in[0,T]}\frac{\ss t}{\aa(t^{\ff 1 2})}\leq \frac{\sqrt{T}\vee 1}{\alpha (1)}$ due to \eqref{AA2} .
\end{proof}

\beg{proof}[Proof of Theorem \ref{LHI}] Let $k=2$. According to \cite[Theorem 2.5]{FYW2} for $D=\R^d$, see also \cite[Theorem 4.1]{YZ}, {\bf (A)} implies the following log-Harnack inequality for some constant $c_0>0$ and any $\gg\in\scr P_2$:
$$P_t^\gg \log f(x)\le \log P_t^\gg f(y)+ \ff{c_0}t |x-y|^2,\ \ x,y\in\R^d, t\in (0,T], f\in \B_b^+(\R^d).$$
Then by \cite[(4.13)]{FYW2}, see also \cite[Theorem 2.1]{HL}, it suffices to find a constant $c>0$ such that
 \beq\label{RR} \beg{split}
& \sup_{t\in (0,T]}\log \E[|R_t^{\gg,\tt\gg}|^2]\le c  \W_2(\gamma,\tilde{\gamma})^2,\ \  \gg,\tt\gg\in \scr P_2, \end{split}\end{equation}
where
\beg{align*} &R_t^{\gg,\tt\gg}:= \e^{\int_0^t\<\eta_s^{\gg,\tt\gg},\d W_s\> -\ff 1 2\int_0^t |\eta_s^{\gg,\tt\gg}|^2\d s},\\
& \eta_s^{\gg,\tt\gg}:= \big\{\si_s^*(\si_s\si_s^*)^{-1}\big\}(X_s^\gg)\big\{b_s(X_s^\gg, P_s^*\tt\gg)- b_s(X_s^\gg,P_s^*\gg)\big\},\ \ s\le t\le T.\end{align*}
Noting that   {\bf (A)} implies
 $$|\eta_s^{\gg,\tt\gg}|^2 \le c_1 \big\{\W_\aa(P_s^\ast\gg,P_s^*\tt\gg)^2+ \W_2(P_s^\ast\gg,P_s^*\tt\gg)^2\big\},\ \ s\in [0,T]$$ for some constant $c_1>0$,  we have
$$ \E[|R_t^{\gg,\tt\gg}|^2] \le   \e^{c_1\int_0^t \{\W_\aa(P_s^\ast\gg,P_s^*\tt\gg)^2+ \W_2(P_s^\ast\gg,P_s^*\tt\gg)^2\}\d s}. $$
Moreover, by \eqref{ASS} and  Lemma \ref{DLP}, there exists a constant $c>0$  such that
 \beg{align*}&\sup_{t\in [0,T]}\int_0^t \big\{\W_\aa(P_s^\ast\gg,P_s^*\tt\gg)^2+ \W_2(P_s^\ast\gg,P_s^*\tt\gg)^2\big\}\d s\le c \W_2(\gamma,\tilde{\gamma})^2.\end{align*}
 Therefore,   \eqref{RR} holds for some constant $c>0$.

\end{proof}

\section{Bismut Formula}

Let $k\in (1,\infty)$ and denote $k^*:=\ff{k}{k-1}$.
In this part, we consider the SDE \eqref{E0} with  $(\tt{\scr P},\tt\rr)= (\scr P_k, \W_k)$, where as in $(A_2)$ the drift $b$ is decomposed as
\beq\label{BB} b_t(x,\nu)= b_t^{(0)}(x)+ b_t^{(1)}(x, \nu),\ \ t\in [0,T], x\in\R^d, \ \nu\in  \scr P_k.\end{equation}
 We aim to establish Bismut type formula for the intrinsic derivative of $\scr P_k\ni\mu\mapsto P_tf(\mu)$
 for bounded measurable functions $f$ on $\R^d$, by only assuming that the extrinsic derivative
 $D^E b_t(x,\mu)(z)$ of the drift
 has a half-Dini continuity in $z\in\R^d$.

 To this end, we first recall the notions of intrinsic and extrinsic derivatives which go back to  \cite{AKR}, see     \cite{BRW20} and \cite{RW212}.

 \begin{defn}\label{ind} Let  $f\in C(\scr P_k;\BB)$ for a Banach space $\BB$.
The function $f$ is called intrinsically differentiable at a point $\mu\in\scr P_k$, if
 $$T_{\mu,k}:=L^k(\R^d\to\R^d;\mu)\ni\phi\mapsto D_\phi^If(\mu):= \lim_{\vv\downarrow 0} \ff{f(\mu\circ(id+\varepsilon\phi)^{-1})-f(\mu)}{\vv}\in\BB
$$ is a well  defined  bounded linear operator. In this case,   the norm of    the intrinsic derivative $D^I f(\mu)$  is given by
$$\|D^If(\mu)\|_{L^{k^*}(\mu)} :=\sup_{\|\phi\|_{L^k(\mu)}\le 1} \|D^I_\phi f(\mu)\|_\BB.$$
The function
  $f$   is called  intrinsically differentiable 
  on $\scr P_k$, if it is so at any $\mu\in \scr P_k$.
 \end{defn}
 Next, we recall  the (convexity) extrinsic derivative, see  e.g.  \cite[Definition 1.2]{RW212}.

   \begin{defn}\label{exd}
A real function $f$ on $\scr P_k$ is called extrinsically differentiable on $\scr P_k$ with derivative $D^Ef$ if
$$D^Ef(\mu)(x)=\lim_{\vv\downarrow 0} \ff{f((1-\varepsilon)\mu+\varepsilon\delta_x)-f(\mu)}{\vv}\in\R$$
exists for all $(x,\mu)\in\R^d\times\scr P_k$. When $f=(f^1,f^2,\cdots,f^d)$ is an $\R^d$-valued function on $\scr P_k$, we denote $D^Ef=(D^Ef^1,D^Ef^2,\cdots,D^Ef^d)$.
 \end{defn}
 \subsection{Main result  }
We will establish a Bismut formula for the intrinsic derivative of $P_tf$ under the following assumption.

\beg{enumerate} \item[{\bf (B)}]   Let $k\in (1,\infty)$ and let  $b$   in \eqref{BB}.
\item[$(B_1)$] $b^{(0)}$ and $\si$  satisfy the corresponding conditions in {\bf (A)}.
\item[$(B_2)$] For any $t\in[0,T], y\in\R^d$, $b_t^{(1)}(y,\cdot)$ is extrinsically differentiable in $\scr P_k$ with the extrinsic derivative $D^E b_t^{(1)}(y,\nu)(z)$ being continuous in $(y,\nu,z)\in\R^d\times \scr P_k \times\R^d$. Moreover, there exists $\aa\in\scr A$ with $\alpha\leq c_0(1+|\cdot|^{k-1})$ for some $c_0>0$ such that
\begin{align*}& |D^E b_t^{(1)}(y,\nu)(z)-D^E b_t^{(1)}(y,\nu)(\bar{z})|\leq \aa(|z-\bar{z}|),\\ &\qquad\quad z,\bar{z}\in\R^d,t\in[0,T], y\in\R^d,\nu\in\scr P_k.
\end{align*}
\item[$(B_3)$] For any $t\in[0,T]$, $\nu\in\scr P_k$, $b_t^{(1)}(\cdot,\nu)$ is differentiable and there exists a constant $\tilde{K}>0$   such that
 \beg{align*}&|b_t^{(1)}(0, \delta_0)|\le \tilde{K},\ \ |\nn b_t^{(1)}(y,\nu)| \le \tilde{K}, \ \ (t,y,\nu)\in [0,T]\times  \R^d\times\scr P_k.\end{align*}
  \end{enumerate}

 As indicated in Introduction that existing results on Bismut type formulas for the intrinsic derivative
 of $P_tf(\mu)$ are established under upper bound conditions on the $L$-derivative of $b_t(y,\nu)$ in $\nu$.
 Noting that under a mild condition, the $L$-derivative equals to the gradient of the extrinsic derivative,
  so the above condition on the $\aa$-continuity of $D^E b_t^{(1)}(y,\nu)(z)$ in $z$ is much weaker.
 To see this, we present below a simple example.

\begin{exa} Let $\aa(s)=s^\vv$ for some $\vv\in (0,1\wedge (k-1))$. Let $g: \R^d\times\R^d\to\R^d$ satisfy
$$|g(y,z)-g(y,\bar{z})|\leq \alpha(|z-\bar{z}|),\ \ |\nabla g(\cdot,z)|\leq K,\ \ y,z,\bar{z}\in\R^d$$
 for some $K\geq 0$.
Let $b_t^{(1)}(y,\nu)=\int_{\R^d}g(y,z)\nu(\d z)$. By Definition \ref{exd}, it holds $$D^Eb_t^{(1)}(y,\nu)(z)=g(y,z)-\int_{\R^d}g(y,z)\nu(\d z),\ \ y\in\R^d,\nu\in\scr P_k, z\in\R^d.$$
However, by Definition \ref{ind}, $b_t^{(1)}(y,\nu)$ is not intrinsically differentiable in $\nu$. In fact, since $g(y,z)$ is not differentiable in $z$, for any $y\in\R^d$, $\nu\in\scr P_k$, $\phi\in L^k(\R^d\to\R^d;\nu) $, the limit
$$\lim_{r\downarrow 0}\frac{\int_{\R^d}g(y,z+r \phi(z))\nu(\d z)-\int_{\R^d}g(y,z)\nu(\d z)}{r}$$
does not exist.
Moreover, it holds
\begin{align*}&|D^Eb_t^{(1)}(y,\nu)(z)-D^Eb_t^{(1)}(\bar{y},\bar{\nu})(\bar{z})|\\
&=|g(y,z)-g(\bar{y},\bar{z})|+\left|\int_{\R^d}g(y,z)\nu(\d z)-\int_{\R^d}g(\bar{y},z)\bar{\nu}(\d z)\right|\\
&\leq \alpha(|z-\bar{z}|)+2K|y-\bar{y}|+\W_\alpha(\nu,\bar{\nu}),\ \ y,\bar{y}\in\R^d,\nu,\bar{\nu}\in\scr P_k, z,\bar{z}\in\R^d.
\end{align*}
Note that Jensen's inequality implies that
$$\W_\alpha(\mu,\nu)\leq \inf_{\pi\in\C(\mu,\nu)}\int_{\R^d\times\R^d}\alpha(|x-y|)\pi(\d x, \d y)\leq \alpha(\W_1(\mu,\nu))\leq \alpha(\W_k(\mu,\nu)),\ \ \mu,\nu\in\scr P_k.$$
So, $D^E b_t^{(1)}(y,\nu)(z)$ is continuous in $(y,\nu,z)\in\R^d\times \scr P_k \times\R^d$. Finally, by the dominated convergence theorem, we have
$$\nn b_t^{(1)}(\cdot,\nu)=\int_{\R^d}\nabla g(\cdot,z)\nu(\d z),\ \ \nu\in\scr P_k.$$
Therefore, $b^{(1)}$ satisfies $(B_2)$-$(B_3)$.
\end{exa}
  Since {\bf (B)} implies  {\bf (A)},    as explained before that under this assumption  \eqref{E0}  is well-posed for distributions in $\scr P_k.$

For $\mu\in \scr P_k$,
 consider the decoupled SDE
\beq\label{DSo'}\begin{split} &\d X_t^{x,\mu}= \big\{b_t^{(0)}(X_t^{x,\mu})+ b_t^{(1)}(X_t^{x,\mu},  P_t^*\mu) \big\}\d t+ \si_t(X_t^{x,\mu})\d W_t,\\
&\qquad\qquad\qquad X_0^{x,\mu}=x, \ \ t\in [0,T].
\end{split}\end{equation}
Let
$$\B_{k,b}(\R^d):=\bigg\{f:\ \ff{f}{1+|\cdot|^k}\in\B_b(\R^d)\bigg\}.$$
We first give a lemma on Bismut formula of $P_t^\mu f$ for $f\in\B_{k,b}(\R^d)$.

\begin{lem}\label{poa} Let $\si$ and $b$ satisfy  {\bf (B)}. Then
 for any  $v\in \R^d, \gamma\in \scr P_k, x\in\R^d$, the limit
\beg{align*} &\nn_v X_t^{x,\gamma}:=\lim_{\vv\downarrow 0} \ff{X_t^{x+\varepsilon v, \gamma}-X_t^{x,\gamma}}\vv,\ \
 \  \ t\in [0,T]\end{align*}
exists in $L^p(\OO\to C([0,T];\R^d);\P)$ for any $p\ge 1$, and  there exists a constant $c_p>0$ such that
\beg{equation}\label{EX}  \E\Big[\sup_{t\in [0,T]} |\nn_v X_t^{x,\gamma}|^p\Big]\le c_p |v|^p,\ \ v\in \R^d, \gamma\in \scr P_k, x\in\R^d. \end{equation}
Moreover, the Bismut formula for $P^{\gamma}_t$ holds:
\beq\label{BS00}\beg{split}&\nabla_v P^{\gamma}_tf (x)= \mathbb{E}\left[ f(X^{x,\gamma}_t)\int_0^t\frac{1}{t}\big\<\zeta_s  (X^{x,\gamma}_s) \nn_{v}  X^{x,\gamma}_s,\d W_s\big\>\right],\\
&\qquad
     \zeta_s:= \si_s^*(\si_s\si_s^*)^{-1},\ \ \ f\in \B_{k,b}(\R^d), \ x,v\in\R^d,\gamma\in\scr P_k,t\in(0,T].\end{split}\end{equation}
\end{lem}

\begin{proof} By \cite[Theorem 2.1]{FYW3} for $\beta_s=\frac{s}{t}$, {\bf (B)} implies   \eqref{EX} and \eqref{BS00} for $f\in \B_b(\R^d)$. To deduce \eqref{BS00}  for any $f\in \B_{b,k}(\R^d),$ let
$$f_n:= [(-n)\lor f]\land n,\ \ n\ge 1.$$
By \eqref{ES00'}, \eqref{EX} and the boundedness of $\zeta_s$, we find  constants $c_0,c_1(\gg)>0$  such that
\beq\label{C00} \beg{split} &\E\bigg[\big(1+|X_t^{x+r v,\gg}|^k\big)\bigg|\int_0^t \big\<\zeta_s  (X^{x+rv,\gamma}_s) \nn_{v}  X^{x+rv,\gamma}_s,\d W_s\big\>\bigg|\bigg]\\
&\le c_0\ss t  \Big(\E\big[2+|X_t^{x+r v,\gg}|^{2k}\big]\Big)^{\ff 1 2} \Big(\E\big[\sup_{s\in[0,T]}\big|\nn_v X^{x+rv,\gg}_s\big|^2\big]\Big)^{\ff 1 2}\\
&\le c_1(\gg)\ss t |v|\big(1+  |x|^k+|v|^k\big),\ \ t\in (0,T], x,v\in\R^d, r\in[0,1].\end{split}\end{equation}
By \eqref{BS00} for $f_n$ in place of $f$, we obtain
\beg{align*}&\ff{ P^{\gamma}_tf_n (x+\vv v)- P^{\gamma}_tf_n (x)}\vv \\
&= \ff 1 \vv \int_0^\vv
\mathbb{E}\left[ f_n(X^{x+ rv,\gamma}_t)\int_0^t\frac{1}{t}\big\<\zeta_s  (X^{x+rv,\gamma}_s) \nn_{v}  X^{x+rv,\gamma}_s,\d W_s\big\>\right]\d r.\end{align*}
Since $f\in \B_{k,b}(\R^d)$, by \eqref{ES00'}, \eqref{EX} and \eqref{C00}, we may apply the dominated convergence theorem such that the above formula with $n\to\infty$ implies
 \beq\label{BS11}\beg{split}&\ff{ P^{\gamma}_tf (x+\vv v)- P^{\gamma}_tf (x)}\vv \\
&= \ff 1 \vv \int_0^\vv
\mathbb{E}\left[ f(X^{x+ rv,\gamma}_t)\int_0^t\frac{1}{t}\big\<\zeta_s  (X^{x+rv,\gamma}_s) \nn_{v}  X^{x+rv,\gamma}_s,\d W_s\big\>\right]\d r,\\
&\qquad f\in \B_{k,b}(\R^d), \vv>0, \ x,v\in\R^d,\gamma\in\scr P_k,t\in(0,T].\end{split}\end{equation}
Note that \eqref{BS00} for $f_n$ in place of $f$ yields
 \beg{align*}&\lim_{n\to\infty} \limsup_{\vv\to 0}\bigg| \ff{ P^{\gamma}_tf_n (x+\vv v)- P^{\gamma}_tf_n (x)}\vv-\mathbb{E}\left[ f(X^{x,\gamma}_t)\int_0^t\frac{1}{t}\big\<\zeta_s  (X^{x,\gamma}_s) \nn_{v}  X^{x,\gamma}_s,\d W_s\big\>\right]\bigg| \\
&=\lim_{n\to\infty}\bigg|\mathbb{E}\left[ f_n(X^{x,\gamma}_t)\int_0^t\frac{1}{t}\big\<\zeta_s  (X^{x,\gamma}_s) \nn_{v}  X^{x,\gamma}_s,\d W_s\big\>\right]\\
&\qquad\qquad\quad-\mathbb{E}\left[ f(X^{x,\gamma}_t)\int_0^t\frac{1}{t}\big\<\zeta_s  (X^{x,\gamma}_s) \nn_{v}  X^{x,\gamma}_s,\d W_s\big\>\right]\bigg|=0,\end{align*}
where the last step follows from  the dominated convergence theorem due to $f\in \B_{k,b}(\R^d)$ and \eqref{C00}. This together with \eqref{BS11} for $f-f_n$ in place of $f$ implies
 \beq\label{ER0} \beg{split}&\limsup_{\vv\to 0}\bigg|\ff{ P^{\gamma}_tf (x+\vv v)- P^{\gamma}_tf (x)}\vv-
\mathbb{E}\left[ f(X^{x,\gamma}_t)\int_0^t\frac{1}{t}\big\<\zeta_s  (X^{x,\gamma}_s) \nn_{v}  X^{x,\gamma}_s,\d W_s\big\>\right]\bigg|\\
&\leq \lim_{n\to\infty}\limsup_{\vv\to 0}\bigg|\ff{ P^{\gamma}_t(f-f_n) (x+\vv v)- P^{\gamma}_t(f-f_n) (x)}\vv
\bigg|\\
&+ \lim_{n\to\infty}\limsup_{\vv\to 0}\bigg|\ff{ P^{\gamma}_tf_n (x+\vv v)- P^{\gamma}_tf_n (x)}\vv\\
&\qquad\qquad\qquad\qquad\quad-\mathbb{E}\left[ f(X^{x,\gamma}_t)\int_0^t\frac{1}{t}\big\<\zeta_s  (X^{x,\gamma}_s) \nn_{v}  X^{x,\gamma}_s,\d W_s\big\>\right]\bigg|\\
&\leq\lim_{n\to\infty} \limsup_{\vv\to 0} \ff 1 {t\vv}\int_0^\vv \E \bigg| (f_n-f)(X_t^{x+rv,\gg})\\
&\qquad\qquad\qquad\qquad\qquad\quad\times    \int_0^t
\big\<\zeta_s  (X^{x+rv,\gamma}_s) \nn_{v}  X^{x+rv,\gamma}_s,\d W_s\big\>\bigg|\d r.
 \end{split} \end{equation}
  Since $f\in \B_{k,b}(\R^d)$ implies
  $$|(f_n-f)(x)|\le c (1+|x|^k) 1_{\{c(1+|x|^k)\ge n\}},\ \ n\ge 1$$ for some constant $c>0$,
 by the same reason leading to \eqref{C00}, we find constants $\tilde{c}, c_2(\gg)>0$ such that
\beg{align*} &\sup_{r\in [0,1]} \E \bigg| (f_n-f)(X_t^{x+rv,\gg})   \int_0^t
\big\<\zeta_s  (X^{x+rv,\gamma}_s) \nn_{v}  X^{x+rv,\gamma}_s,\d W_s\big\>\bigg|\\
&\le \tilde{c}\ss t |v| \sup_{r\in [0,1]}\Big(\E \big[1+|X_t^{x+rv,\gg}|^{4k}\big]\Big)^{\ff 1 2} n^{-1}\\
&\le c_2(\gg)\ss t |v| \big(1+|x|^{2k}+|v|^{2k}\big)n^{-1}.\end{align*}
 Therefore, \eqref{BS00} follows form \eqref{ER0}.
\end{proof}

 To state the Bismut formula for $P_tf$, we introduce the quantity $I_t^f$:
for fixed  $t\in (0,T]$,  let
\beq\label{ZI}  \beg{split}
   I_t^{f}(\mu,\phi)&:=
 \ff 1 t   \int_{\R^d} \E\bigg[f(X_t^{x,\mu}) \int_0^t \big\<\zeta_s(X_s^{x,\mu} ) \nn_{\phi (x)}X_s^{x,\mu},  \d W_s\big\> \bigg]\mu(\d x),\\
     & s\in [0,t],\mu\in \scr P_k, \phi\in T_{\mu,k},f\in \B_{k-1,b}(\R^d).\end{split} \end{equation}
  By  {\bf(B)} and \eqref{EX}, we find a constant $c>0$ such that
\beq\label{BTT}  |I_t^f(\mu,\phi)|\le   \ff c{\ss t}  \big(P_t|f|^{k^*}(\mu)\big)^{\ff 1 {k^*}} \|\phi\|_{L^k(\mu)},\ \ \mu\in \scr P_k, \phi\in T_{\mu,k}, f\in \B_{k-1,b}(\R^d).\end{equation}

Next, let   $X_0^\mu$ be $\F_0$-measurable such that $\L_{X_0^\mu}=\mu$, and let $X_t^\mu$ solve \eqref{E0} with initial value $X_0^\mu$. For any $\vv\ge 0$, denote
$$\mu_\vv:= \mu\circ(id+\vv\phi)^{-1},\ \ X_0^{\mu_\vv}:= X_0^\mu+\vv\phi(X_0^\mu).$$
Let $X_t^{\mu_\vv}$ solve \eqref{E0} with initial value $X_0^{\mu_\vv}$. So,
$$X_t^{\mu}=X_t^{\mu_0},\ \ P_t^*\mu_\vv=\L_{X_t^{\mu_\vv}},\ \ t\in [0,T], \vv\ge 0.$$
Now, we present the main result of this part.
\beg{thm}\label{TA2'} Assume   {\bf(B)} and let $\zeta_s$ and $I_t^f$ be in \eqref{BS00} and $\eqref{ZI}$ respectively. Then   the following assertions hold.
 \beg{enumerate}  \item[$(1)$]   For any $t\in (0,T]$, $y\in\R^d$, $\nu\in\scr P_k$,  $P_t[D^E b_t^{(1)}(y,\nu)(\cdot)](\mu)$ is intrinsically differentiable on $\mu\in\scr P_k$, and there exists a constant $c>0$ such that
\begin{align*}   \sup_{(y,\nu)\in\R^d\times \scr P_k}\|D^IP_t[D^E b_t^{(1)}(y,\nu)(\cdot)](\mu)\|_{L^{\frac{k}{k-1}}(\mu)} \le \ff{c\,\aa(t^{\ff 1 2})}{\ss t},\ \ t\in (0,T],\mu\in\scr P_k.\end{align*}
\item[$(2)$] For any $t\in (0,T]$ and $f\in \B_{k-1,b}(\R^d)$, $P_t f$ is intrinsically differentiable on $\scr P_k$. Moreover, for any $\mu\in \scr P_k$ and $\phi\in T_{\mu,k}$,
\beq\label{BSMI}   \beg{split} D_\phi^I  P_tf(\mu)&= I_t^f(\mu,\phi)\\
&+ \int_{\R^d}\E\bigg[f(X_t^{x,\mu})\int_0^t \Big\<\zeta_s (X_s^{x,\mu}) N_s(\mu,\phi),\ \d W_s\Big\>\bigg]\mu(\d x),\\
&N_s(\mu,\phi):=  \big\{D_\phi^I P_s [D^Eb_s^{(1)}(y, \nu)(\cdot)](\mu) \big\}|_{y=X_s^{x,\mu},\nu=P_s^\ast\mu},  \end{split}
  \end{equation}

  where $X_t^{x,\mu}$ solves $\eqref{DSo'}$ with initial value $x\in\R^d$.
  \end{enumerate} \end{thm}
By \eqref{BTT}-\eqref{BSMI},  we find  a constant $c>0$ such that
 \beg{align*}   \|D^I P_tf(\mu)\|_{L^{k^*}(\mu)} \le c\ff{\{P_t|f|^{k^*}(\mu)\}^{\ff 1 {k^*}}}{\ss t},\ \
  t\in (0,T],   f\in \B_{k-1,b}(\R^d),  \mu\in \scr P_k.\end{align*}

To explain the main idea in the proof of Theorem \ref{TA2'},    we first figure out a sketch. By the definition of the intrinsic derivative, we intend to calculate for any $f\in \B_{k-1,b}(\R^d)$,
\beq\label{D1} D_\phi^I P_t f(\mu):=\lim_{\vv\downarrow 0}\ff{P_t f(\mu_\vv)-P_tf(\mu)}\vv= \lim_{\vv\downarrow 0}\ff{\E[ f(X_t^{\mu_\vv})-f(X_t^\mu)]}\vv.\end{equation}
To this end, for any $\mu\in \scr P_k, x\in\R^d$,  recall that $X_t^{x,\mu}$ solves the decoupled SDE \eqref{DSo'},
and
\beg{align*} &P_t^\mu f(x)=\E[f(X_t^{x,\mu})],\ \ \ x\in\R^d.
\end{align*}
Define
\begin{align*}P_t^\mu f(\tt\mu):= \int_{\R^d} P_t^\mu f\d\tt\mu,\ \ t\ge 0, f\in\B_{k-1,b}(\R^d), \mu,\tt\mu\in \scr P_k.\end{align*}
For $\vv\ge 0$, let
 $X_t^{\mu_\vv,\mu} $ be the solution of \eqref{DSo'} with initial value
 $X_0^\vv$, i.e,
 \begin{equation*} \beg{split} \d X_t^{\mu_\vv,\mu}=\big\{b_t^{(0)}(X_t^{\mu_\vv,\mu}) + &b_t^{(1)}\big(X_t^{\mu_\vv,\mu},P_t^\ast\mu\big)\big\}\d t
   +\si_t(X_t^{\mu_\vv,\mu})\d W_t,\\
  &  t\in [0,T], \ X_0^{\mu_\vv,\mu}=X_0^\vv.\end{split}\end{equation*}  Then $X_t^{\mu_\vv,\mu_\vv}$ solves \eqref{E0} with initial value $X_0^\vv$, so that
 \begin{equation*}P_t f(\mu_\vv)= P_t^{\mu_\vv} f(\mu_\vv)= \E[f(X_t^{\mu_\vv,\mu_\vv})],\ \ \vv\ge 0, t\in [0,T], f\in \B_{k-1,b}(\R^d).\end{equation*}
Noting that $\mu_0=\mu$,  \eqref{D1} reduces to
\beq\label{BS0} \beg{split} &D_\phi^I P_t f(\mu)=\lim_{\vv\downarrow 0}\ff{P_t^{\mu_\vv} f(\mu_\vv)- P_t^\mu f(\mu)}\vv\\
&= \lim_{\vv\downarrow 0}\Big\{\ff{P_t^{\mu} f(\mu_\vv)- P_t^\mu f(\mu)}\vv+\ff{P_t^{\mu_\vv}f(\mu_\vv)- P_t^{\mu}f(\mu_\vv)}\vv\Big\}.\end{split}\end{equation}
So, to calculate $D_\phi^IP_tf(\mu),$
we only need to study the limits of
$$J_1f(t,\vv):=\ff{P_t^{\mu} f(\mu_\vv)- P_t^\mu f(\mu)}\vv,\ \   J_2f(t,\vv):= \ff{P_t^{\mu_\vv}f(\mu_\vv)- P_t^{\mu}f(\mu_\vv)}\vv.$$

By Lemma \ref{poa},   for any $t\in (0,T],\vv\ge 0$ and $f\in \B_{k-1,b}(\R^d)$, we have
\begin{equation*}  \beg{split} & \ff{\d}{\d\vv} P_t^{\mu}f(\mu_\vv):=  \lim_{r\downarrow 0} \ff{P_t^\mu f(\mu_{\vv+r})- P_t^{\mu}f(\mu_\varepsilon)}r\\
&=   \int_{\R^d} \E\bigg[f(X_t^{x+\vv\phi(x),\mu}) \frac{1}{t}\int_0^t\big\<\zeta_s(X_s^{x+\vv\phi(x),\mu} ) \nn_{\phi (x)}X_s^{x+\vv\phi(x),\mu},  \d W_s\big\> \bigg]\mu(\d x).\end{split}\end{equation*}
In particular,
\beq\label{BS1'} \lim_{\vv\downarrow 0} J_1f(t,\vv)= \lim_{\vv\downarrow 0} \ff{P_t^{\mu}f(\mu_\vv)- P_t^{\mu}f(\mu)}\vv=I_t^f(\mu,\phi),\ \ t\in (0,T].\end{equation}
 Consequently,  it remains to prove
$$\lim_{\vv\to0} J_2f(t,\vv)= \int_{\R^d}\E\bigg[f(X_t^{x,\mu})\int_0^t \Big\<\zeta_s (X_s^{x,\mu}) N_s(\mu,\phi),\ \d W_s\Big\>\bigg]\mu(\d x)$$
for $N_s(\mu,\phi)$ defined in \eqref{BSMI}, which involves in $D_\phi^I \{P_s[ D^Eb_s^{(1)}(y,\nu)(\cdot)](\mu)\}$. Therefore, we will first study $D_\phi^I \{P_s[ D^Eb_s^{(1)}(y,\nu)(\cdot)](\mu)\}$.

 Recall that $\tilde{\alpha}$ is defined in \eqref{tal}. For any $\v\in \scr B_b([0,T]\times \R^d\times\scr P_k;\R^d)$, the set of bounded and measurable $\R^d$-valued functions on $[0,T]\times \R^d\times\scr P_k$ and $t\in[0,T], y\in\R^d,\nu\in\scr P_k$, we write
\begin{equation}\label{MYt} \beg{split}
\tt I_t^{\v}(y,\nu)&:=     \int_{\R^d}\E\bigg[D^Eb_t^{(1)}(y, \nu)(X_t^{x,\mu}) \\
&\qquad\qquad\quad \times\int_0^t  \ff{ \aa(s^{\ff 1 2})}{  \{\tt\aa(s^{\ff 1 2}) s\}^{\ff 1 2}}\Big\<\zeta_s(X_s^{x,\mu}) \v_s(X_s^{x,\mu},P_s^\ast\mu),\ \d W_s\Big\>\bigg]\mu(\d x).
 \end{split}\end{equation}
By $(B_2)$, for any $t\in[0,T], y\in\R^d,\nu\in\scr P_k$, we have
\begin{align*}|D^Eb_t^{(1)}(y, \nu)(\cdot)|&\leq \alpha(|\cdot|)+|D^Eb_t^{(1)}(y, \nu)(0)|\\
&\leq c_0(1+|\cdot|^{k-1})+|D^Eb_t^{(1)}(y, \nu)(0)|.
\end{align*}
So, $I_t^f(\mu,\phi)$ for   $f=D^Eb_t^{(1)}(y, \nu)(\cdot)$ in \eqref{ZI} is well-defined, and we denote
\beq\label{IVa}  \beg{split}
   I_t^{\mu,\phi}(y,\nu)&:=
 \ff 1 t   \int_{\R^d} \E\bigg[D^Eb_t^{(1)}(y, \nu)(X_t^{x,\mu})\\
  &\qquad\quad\times \int_0^t \big\<\zeta_s(X_s^{x,\mu} ) \nn_{\phi (x)}X_s^{x,\mu},  \d W_s\big\> \bigg]\mu(\d x).\end{split} \end{equation}
Consider  the following   equation for $\v\in \scr B_b([0,T]\times \R^d\times\scr P_k;\R^d)$:
\beq\label{VP}\begin{split}
& \v_t(y,\nu)   =  {\ff{\{t\tt\aa(t^{\ff 1 2})\}^{\ff 1 2}}{\aa(t^{\ff 1 2})}}  \big\{I_t^{\mu,\phi}(y,\nu)+ \tt I_t^{\v}(y,\nu)\big\},\ \ t\in [0,T],y\in\R^d,\nu\in\scr P_k.
\end{split}\end{equation} If this equation has a unique solution, we denote it by $\v_t(y,\nu)=v_t^{\mu,\phi}(y,\nu)$ for $(t,y,\nu)\in [0,T]\times\R^d\times\scr P_k$, to emphasize the dependence on $\mu$ and $\phi$.

In the following two subsections, we  prove the  well-posedness of   \eqref{VP} and establish the formula
\beq\label{00}\begin{split} &D_\phi^I \{P_t [D^Eb_t^{(1)}(y,\nu)(\cdot)]\}(\mu)= {\ff{\aa(t^{\ff 1 2})}{ \{t \tt\aa(t^{\ff 1 2})\}^{\ff 1 2} }}\,v_t^{\mu,\phi}(y,\nu),\ \ t\in(0,T],y\in\R^d,\nu\in\scr P_k.
\end{split}\end{equation}
\ \subsection{Well-posedness of \eqref{VP} }

\beg{lem}\label{LL20} Assume {\bf (B)}.  For any $\mu\in \scr P_k$ and $\phi\in T_{\mu,k}$,  the equation $\eqref{VP} $ has a unique solution, which is denoted by $\{v_t^{\mu,\phi}(y,\nu)\}_{t\in [0,T],y\in\R^d,\nu\in\scr P_k}$, and there exists a constant $c>0$   such that
 \beq\label{EST0} \sup_{\|\phi\|_{L^k(\mu)}\le 1} \sup_{y\in\R^d,\nu\in\scr P_k}|v_t^{\mu,\phi}(y,\nu)|\le    {c\,\ss{\tt\aa( t^{\ff 1 2})}},\ \ \mu\in \scr P_k, t\in [0,T].
 \end{equation}
\end{lem}

\beg{proof}
Let
$$\scr V_0:=  \big\{\v\in \scr B_b([0,T]\times \R^d\times\scr P_k;\R^d):\ \v_0=0\big\},$$ which is a Banach space under the uniform norm.
For $\v\in \scr V_0,$ let
$$\|\v_t\|_\infty=\sup_{y\in\R^d,\nu\in\scr P_k}|\v_t(y,\nu)|,\ \ t\in [0,T] $$ and for any $ t\in[0,T], y\in \R^d,\nu\in\scr P_k$, let
\beq\label{HT} \beg{split} & \{H(\v)\}_t(y,\nu):=  { \ff{\{t\tt\aa(t^{\ff 1 2})\}^{\ff 1 2}}{\aa(t^{\ff 1 2})}}  \big\{I_t^{\mu,\phi}(y,\nu)+ \tt I_t^{\v}(y,\nu)\big\}.
\end{split} \end{equation}
Then  it suffices  to prove
\beg{enumerate} \item[(i)]  The map
$H:\ \scr V_0\to \scr V_0$
 is well-defined and  has a unique fixed point $v^{\mu,\phi}$ which turns out to be the unique solution of $\eqref{VP} $.
 \item[(ii)]  {There exists a constant $c>0$ such that
 \beg{align*}   \sup_{\|\phi\|_{L^k(\mu)}\le 1} \|v_t^{\mu,\phi}\|_\infty\le c  {\,\ss{\tt\aa( t^{\ff 1 2})}},\ \ (t,\mu)\in [0,T]\times\scr P_k.
 \end{align*}
 }
\end{enumerate}
 Next, we will prove (i) and (ii) one by one.

(1) Proof of (i).

(a) We first verify
\beq\label{BHH} H: \scr V_0\to \scr V_0.\end{equation}

Recall that $\theta^{\lambda,\gamma}$ and $\tilde{\theta}^{\lambda,\gamma}$ are defined in \eqref{acd}-\eqref{thv1}.
Since {\bf (B)} implies {\bf(A)}, we conclude that \eqref{soi} still holds.

 By $[D^E b_t^{(1)}(y,\nu)(\cdot)]_\aa\le 1$ due to $(B_2)$,  \eqref{aas} in Lemma \ref{LN} for $p=2$ and   { $z=\tilde{\theta}^{\lambda,\mu}_t(x)$, \eqref{EX}, \eqref{soi}, \eqref{AA} and \eqref{IVa},  we find a constant $c_1>0$ such that}
\beq\label{IT1} \begin{split} |I_t^{\mu,\phi}(y,\nu)|
&\leq \frac{c_1}{\sqrt{t}}\int_{\R^d} {\alpha\big(t^{\frac{1}{2}}\big)|\phi(x)|\mu(\d x)}\\
& {\leq \frac{c_1}{\sqrt{t}}\alpha\big(t^{\frac{1}{2}}\big)\|\phi\|_{L^k(\mu)},\ \ t\in (0,T], \mu\in\scr P_k, \phi\in T_{\mu,k}}, y\in\R^d,\nu\in\scr P_k.
\end{split}\end{equation}

So,     by \eqref{MYt}, $(B_2)$, \eqref{aas} in Lemma \ref{LN} for $p=2$ and   { $z=\tilde{\theta}^{\lambda,\mu}_t(x)$}, \eqref{AA} and \eqref{soi}, we find a constant $c_2>0$ such that
\beq\label{tid}\begin{split}
   & |\tt I_t^{\v}(y,\nu)|  \le c_2    {\alpha\big(t^{\frac{1}{2}}\big)}\bigg(\int_0^t  {\ff{\aa(s^{\ff 1 2})^2}{s\tt\aa(s^{\ff 1 2})}}\|\v_s\|^2_\infty\d s\bigg)^{\ff 1 2},\ \ y\in\R^d,\nu\in\scr P_k.
\end{split}\end{equation}
 Combining this with \eqref{HT} and \eqref{IT1}, we find  a constant $c_3>0$     such that
  {\beq\label{UPP}\beg{split} \|\{H(\v)\}_t\|_\infty\le  &\,c_3    \|\phi\|_{L^k(\mu)}  \ss{\tt\aa\big(t^{\ff 1 2}\big)} +  c_3    \ss{t \tt \aa\big(t^{\ff 1 2}\big)}  \bigg(  \int_0^t  \ff{\aa(s^{\ff 1 2})^2}{s\tt\aa\big(s^{\ff 1 2}\big)}\|\v_s\|^2_\infty\d s\bigg)^{\ff 1 2}.\end{split}\end{equation}}
Then  \eqref{BHH} follows from the fact that  \eqref{tal} implies
\beq\label{C1} \int_0^t \ff{\aa(r s^{\ff 1 2})^2}{s\tt\aa(r s^{\ff 1 2})} \d s= 2\int_0^{rt^{\ff 1 2}}  \ff{\aa(s)^2}{s\tt\aa(s)}\d s= 4 \int_0^{r t^{\ff 1 2}} \tt\aa'(s)\d s=4 \tt\aa(rt^{\ff 12}),\ \ r\ge 0.\end{equation}

(b) We intend to prove that $H$   in (i) has a unique fixed point in $\scr V_0.$ Obviously, for any $\delta>0$, $\scr V_0$ is complete under the metric
\beg{align*} &\rr_{\delta}\big(\v, U\big):= \sup_{t\in [0,T]} \e^{-\delta t} \|\v_t-U_t\|_{\infty},\ \ \v,U\in \scr V_0.\end{align*}
So, it suffices to prove the contraction of
$H$ in $\rr_{\delta}$ for large enough $\delta>0$.

By \eqref{HT}, \eqref{IT1} and \eqref{tid}, we find a constant $c_4>0$ such that
 \beg{align*}&|\{H(\v)\}_t(y,\nu)-\{H(U)\}_t(y,\nu)| =  \ff{\{t\tt\aa\big(t^{\ff 1 2}\big)\}^{\ff 1 2}}{\aa\big(t^{\ff 1 2}\big)}  | \tt I_t^{\v-U}(y,\nu)|\\
& \le c_4   \bigg( \int_0^t  \ff{\aa(s^{\ff 1 2})^2}{s\tt\aa(s^{\ff 1 2})}\|\v_s-U_s\|^2_\infty\d s\bigg)^{\ff 1 2},\ \ V,U\in \scr V_0, t\in [0,T].\end{align*}
Combining this with \eqref{C1}, we conclude that $H$  is contractive in the complete metric space $(\scr V_0, \rr_{\delta})$
 for large enough $\delta>0$, and hence  has a unique fixed point denoted by $v^{\mu,\phi}$.

(2) Proof of (ii). By \eqref{UPP}  and noting that   $H(v^{\mu,\phi})=v^{\mu,\phi}$, we derive
\beg{align*} &\|v_t^{\mu,\phi}\|_\infty^2\le   2c_3^2  \tt\aa\big(t^{\ff 1 2}\big) \|\phi\|_{L^k(\mu)}^2  +  2c_3^2   t \tt \aa\big(t^{\ff 1 2}\big)\int_0^t  \ff{\aa(s^{\ff 1 2})^2}{s\tt\aa(s^{\ff 1 2})}\|v_s^{\mu,\phi}\|^2_\infty\d s,\ \ t\in [0,T].\end{align*}
Combining this with \eqref{C1} and  Gronwall's inequality, we find a constant $c_5>0$ such that for any $t\in [0,T]$,
\beg{align*} &   \|v_t^{\mu,\phi}\|_{\infty}\le  c_5\|\phi\|_{L^k(\mu)} \,\ss{\tt\aa( t^{\ff 1 2})},\ \ \mu\in \scr P_k, \ \phi\in T_{\mu,k}.\end{align*}
 This proves (ii).
\end{proof}
\subsection{ {Proof of Theorem \ref{TA2'}}}
By Lemma \ref{LL20},   the proof of Theorem \ref{TA2'}(1) is completed by the following lemma.
\beg{lem}\label{LL2} Assume {\bf (B)}. Then for any    $\mu\in \scr P_k$, $\phi\in T_{\mu,k}$, the function $h:(0,T]\times\R^d\times \scr P_k\to \R^d$ defined by $$h_t(y,\nu):=
D_\phi^I\big\{ P_t [D^Eb_t^{(1)}(y, \nu)(\cdot)](\mu) \big\},\ \ t\in(0,T], y\in\R^d, \nu\in\scr P_k $$
 exists in $\scr B((0,T]\times\R^d\times \scr P_k; \R^d)$ such that $\eqref{00}$ holds. Consequently,
  there exists  a constant $c>0$ such that for any $ \mu\in \scr P_k$,
\begin{align*}\sup_{\|\phi\|_{L^k(\mu)}\le 1} \bigg\{\sup_{t\in (0,T]}\ff{\ss t}{\aa(t^{\ff 1 2})}\sup_{y\in \R^d,\nu\in\scr P_k}|D_\phi^I\big\{ P_t [D^Eb_t^{(1)}(y, \nu)(\cdot)](\mu) \big\}| \bigg\}
 \le c.
 \end{align*}
 \end{lem}

\beg{proof} (a) By Lemma \ref{LL20}, it suffices to prove \eqref{00}. For simplicity,
for any $t\in[0,T], y\in\R^d,\nu\in\scr P_k$, let
\beq\label{VW}U_t(y,\nu,z):=D^E b_t^{(1)}(y, \nu)(z), \ \ z\in\R^d.\end{equation}
Moreover, simply denote
\beg{align*} &v_t^{\vv}(y,\nu):=v_t^{\vv,1}(y,\nu)+ v_t^{\vv,2}(y,\nu),\\
&v_t^{\vv,1}(y,\nu):=\ff{P_t^{\mu}U_t(y,\nu,\cdot)(\mu_\vv)- P_t^{\mu}U_t(y,\nu,\cdot)(\mu)}\vv,\\
&v_t^{\vv,2}(y,\nu):=  \ff{P_t^{\mu_\vv}U_t(y,\nu,\cdot)(\mu_\vv)- P_t^{\mu}U_t(y,\nu,\cdot)(\mu_\vv)}\vv. \end{align*}
Next, for $v_t^{\mu,\phi}$  in Lemma \ref{LL20}, let
\beq\label{W1'} \beg{split}&\hat v_t(y,\nu):=  {\ff{\aa(t^{\ff 1 2})}{ \{\tt\aa(t^{\ff 1 2})  t\}^{\ff 1 2} }}v_t^{\mu,\phi}(y,\nu),\ \ t\in(0,T],y\in\R^d,\nu\in\scr P_k,
\end{split} \end{equation}
and
\begin{align}\label{v12}
\hat v_t^1(y,\nu):=I_t^{U_t(y,\nu,\cdot)}(\mu,\phi),\ \ \hat v_t^2(y,\nu):= \hat v_t(y,\nu)- I_t^{U_t(y,\nu,\cdot)}(\mu,\phi).
\end{align}
Noting that
\begin{align}\label{ULT}
\nonumber&P_t^\mu U_t(y,\nu,\cdot)(\mu_\vv)- P_t^\mu U_t(y,\nu,\cdot)(\mu)\\
&= \int_{\R^d} \Big[P_t^\mu U_t(y,\nu,\cdot)(x+\vv\phi(x))- P_t^\mu U_t(y,\nu,\cdot)(x)\Big]\mu(\d x),
\end{align}
by \eqref{BS00} for $f=U_t(y,\nu,\cdot)$, we obtain
\beq\label{P00} \lim_{\vv\to 0} \big|v_t^{\vv,1}(y,\nu)- \hat v_t^1(y,\nu)\big|=0,\ \ t\in(0,T],y\in\R^d,\nu\in\scr P_k.\end{equation}

Since  \eqref{VP} holds for $V_t=v_t^{\mu,\phi}$,     \eqref{VW}-\eqref{v12} imply
that
\beq \label{v2c} \beg{split} &\hat v_t^2(y,\nu)= \hat v_t(y,\nu)- I_t^{U_t(y,\nu,\cdot)}(\mu,\phi)\\
 & = \int_{\R^d}\E\bigg[ U_t(y,\nu,X_t^{x,\mu}) \\
 &\qquad\qquad\quad \times\int_0^t
   \Big\<\zeta_s(X_s^{x,\mu})[ \hat v_s^2(X_s^{x,\mu},P_s^\ast\mu)+\hat v_s^1(X_s^{x,\mu},P_s^\ast\mu)],  \ \d W_s\Big\>\bigg]\mu(\d x).\end{split}\end{equation}
In view of \eqref{P00}, to prove \eqref{00},  it remains to verify
\beq\label{SFF}  \lim_{\vv\to 0}\sup_{t\in (0,T]} \ff{\ss t}{\aa(t^{\ff 1 2})} | v_t^{\vv,2}(y,\nu) - \hat v_t^{2}(y,\nu)| =0,\ \   y\in\R^d,\nu\in\scr P_k.\end{equation}
  In the following, we first   estimate  $\|v_t^{\vv,i}\|_{\infty}$ and $| v_t^{\vv,i} - \hat v_t^i| (i=1,2)$ in steps (b)-(c), then verify \eqref{SFF} in step (d).

(b) Estimates on $\|v_t^{\vv,i}\|_\infty, i=1,2$ and $\|v_t^{\vv,1}-\hat v_t^1\|_{\infty}$.

 By \eqref{BS11} for
$f= U_t(y,\nu,\cdot)$ and \eqref{ULT}, we obtain
\begin{align}\label{v1a}
\nonumber&v_t^{\vv,1}(y,\nu)=\frac{P_t^\mu U_t(y,\nu,\cdot)(\mu_\vv)- P_t^\mu U_t(y,\nu,\cdot)(\mu)}{\vv}\\
\nonumber&= \frac{1}{\varepsilon}\int_0^\varepsilon \int_{\R^d}\E\bigg[\Big(U_t\big(y,\nu,X_t^{x+r \phi(x),\mu}\big)-U_t\big(y,\nu,\tilde{\theta}^{\lambda,\mu}_t(x+r \phi(x))\big)\Big)\\
&\qquad\qquad\qquad\qquad \times \int_0^t\ff 1 t \big\<\zeta_s(X_s^{x+r \phi(x),\mu})\nn_{\phi(x)} X_s^{x+r\phi(x),\mu},\d W_s\big\>\bigg]\mu(\d x)\d r\\
\nonumber&=\int_0^1 \int_{\R^d}\E\bigg[\Big(U_t\big(y,\nu,X_t^{x+\vv u \phi(x),\mu}\big)-U_t\big(y,\nu,\tilde{\theta}^{\lambda,\mu}_t(x+\vv u \phi(x))\big)\Big)\\
\nonumber&\qquad\qquad\qquad\qquad \times \int_0^t\ff 1 t \big\<\zeta_s(X_s^{x+\vv u \phi(x),\mu})\nn_{\phi(x)} X_s^{x+\vv u\phi(x),\mu},\d W_s\big\>\bigg]\mu(\d x)\d u,
\end{align}
where in the last step, we used the integral transform $r =\vv u$. Similar to \eqref{IT1},
noting that {\bf (B)} implies $[U_t(y,\nu,\cdot)]_\alpha\leq 1$, by \eqref{aas} in Lemma \ref{LN} for $p=2$ and $z=\tilde{\theta}^{\lambda,\mu}_t(x+\varepsilon r \phi(x))$, \eqref{EX}, \eqref{soi} and \eqref{AA}, we   find a constant $c(\mu,\phi)>0$ depending on $\phi,\mu$ such that
\beq\label{W0} \beg{split}&\|v_t^{\vv,1}\|_\infty =\sup_{y,\nu} \ff{|P_t^{\mu}U_t(y,\nu,\cdot)(\mu_\vv)- P_t^{\mu}U_t(y,\nu,\cdot)(\mu)|}\vv\\
&\leq \sup_{y,\nu}\int_0^1 \int_{\R^d}\bigg|\E\bigg[\Big(U_t\big(y,\nu,X_t^{x+\vv r \phi(x),\mu}\big)-U_t\big(y,\nu,\tilde{\theta}^{\lambda,\mu}_t(x+\vv r \phi(x)))\Big)\\
&\qquad\qquad\qquad\qquad \times \int_0^t\ff 1 t \big\<\zeta_s(X_s^{x+ \vv r\phi(x),\mu})\nn_{\phi(x)} X_s^{x+\vv r\phi(x),\mu},\d W_s\big\>\bigg]\bigg|\mu(\d x)\d r\\
&\le \ff{c(\mu,\phi) \aa(t^{\ff 1 2})}{\ss t},\ \ \ \vv\in (0,1], t\in (0,T].\end{split}\end{equation}
This together with \eqref{P00} and \eqref{IT1} implies that for a constant $c(\mu,\phi)>0$
 $$h_{t,1}^{\vv}(y,\nu):=\big\{v_t^{\vv,1}(y,\nu)- \hat{v}_t^{1}(y,\nu)\big\}   \ff{\{t\tt\aa(t^{\ff 1 2})\}^{\ff 1 2}}{\aa(t^{\ff 1 2})} $$ satisfies
 \beq\label{W2} \lim_{\vv\to 0}   |h_{t,1}^{\vv}(y,\nu)|=0,\ \   \sup_{\vv\in (0,1]}\sup_{y,\nu}| h_{t,1}^{\vv}(y,\nu)|\le c(\mu,\phi) \ss{\tt\aa(t^{\ff 1 2})},\ \ t\in (0,T].\end{equation}

Next, we estimate $\|v_t^{\vv,2}\|_\infty.$  Recall that $X_t^{x+\vv\phi(x),\mu}$ solves \eqref{DSo'} with initial value $x+\vv\phi(x)$. For any $ x\in\R^d,s,t\in [0,T]$, let
\beq\label{RTE} \beg{split}
&R_t^{\vv,x}:= \e^{\int_0^t\<\eta_s^{\vv,x},\d W_s\>- \ff 1 2\int_0^t |\eta_s^{\vv,x}|^2\d s},\\
&\eta_s^{\vv,x}:= \zeta_s(X_s^{x+\vv\phi(x),\mu})\\
&\qquad\quad\times\big\{b_s^{(1)}(X_s^{x+\vv\phi(x),\mu}, P_s^\ast\mu_\vv) - b_s^{(1)}(X_s^{x+\vv\phi(x),\mu}, P_s^\ast\mu)\big\}.
\end{split}\end{equation}
By \cite[Lemma 3.2]{RW212}, we have
\begin{align}\label{B-B}
\nonumber&b_t^{(1)}(y,P_t^\ast\mu_\vv)-b_t^{(1)}(y,P_t^\ast\mu)\\
&=\int_0^1\frac{\d }{\d r}b_t^{(1)}(y,(1-r)P_t^\ast\mu+rP_t^\ast\mu_\vv)\d r\\
\nonumber&=\int_0^1\int_{\R^d}D^E b_t^{(1)}(y,(1-r)P_t^\ast\mu+rP_t^\ast\mu_\vv)(z)(P_t^\ast\mu_\vv-P_t^\ast\mu)(\d z)\d r.
\end{align}
Since {\bf (B)} implies {\bf (A)}, Lemma \ref{DLP} holds so that we find constants $c_0,c>0$ such that
\beq\label{EE} \beg{split}  |\eta_s^{\vv,x}|&\le c_0\W_\alpha(P_s^\ast\mu_\vv, P_s^\ast\mu)\leq c\varepsilon\|\phi\|_{L^k(\mu)}\ff{ \aa(s^{\ff 1 2}) }{\ss s}, \ \ s\in [0,T],\ \vv\in [0,1],x\in\R^d. \end{split}\end{equation}
 Then by Girsanov's theorem, for any $x\in\R^d$,
$$ W_t^{\vv,x}:= W_t-\int_0^t\eta_s^{\vv,x} \d s,\ \ s\in [0,T]$$ is a
Brownian motion under $\Q:= R_T^{\vv,x}\P$.
Reformulate  \eqref{DSo'} with $x+\vv\phi(x)$ replacing $x$  as
 \beg{align*} \d X_t^{x+\vv\phi(x),\mu}&=\Big\{b_t^{(0)}(X_t^{x+\vv\phi(x),\mu}) + b_t^{(1)}(X_t^{x+\vv\phi(x),\mu}, P_t^*\mu) \Big\}\d t+ \si_t(X_t^{x+\vv\phi(x),\mu})\d W_t\\
&=\Big\{b_t^{(0)}(X_t^{x+\vv\phi(x),\mu}) + b_t^{(1)}(X_t^{x+\vv\phi(x),\mu}, P_t^*\mu_\vv) \Big\}\d t+ \si_t(X_t^{x+\vv\phi(x),\mu})\d W_t^{\vv,x},\\
&\qquad \ X_0^{x+\vv\phi(x),\mu}=x+\vv\phi(x),\ x\in\R^d.\end{align*}
By the weak uniqueness of \eqref{DSo'} with $\mu=\mu_\vv,$ we get
 \begin{align}\label{B00} \nonumber v_t^{\vv,2}(y,\nu)&=  \ff{P_t^{\mu_\vv}U_t(y,\nu,\cdot)(\mu_\vv)- P_t^{\mu}U_t(y,\nu,\cdot)(\mu_\vv)}\vv\\
& =\ff{\int_{\R^d}[P_t^{\mu_\vv}U_t(y,\nu,\cdot)(x+\vv\phi(x))- P_t^{\mu}U_t(y,\nu,\cdot)(x+\vv\phi(x))]\mu(\d x)}\vv\\
\nonumber &= \ff 1 \vv \int_{\R^d}\E[U_t(y,\nu,X_t^{x+\vv\phi(x),\mu})(R_t^{\vv,x}-1)]\mu(\d x),\ \ t\in [0,T].\end{align}
By  \eqref{EE}, for any $p\ge 1$ there exists  a constant $c (p,\mu,\phi)>0$ such that
\beq\label{EE'}  \E[|R_t^{\vv,x}-1|^p] \le c(p,\mu,\phi)\vv^p \bigg(\int_0^t\ff{ \aa(s^{\ff 1 2})^2 }{s}\d s\bigg)^{\ff p 2},\ \ t\in [0,T], \vv\in [0,1], x\in\R^d.\end{equation}
 Again by \eqref{aas} in Lemma \ref{LN} for $p=2$ and $z=\tilde{\theta}^{\lambda,\mu}_t(x+\vv \phi(x))$, \eqref{B00}, \eqref{EE'}, \eqref{EX}, \eqref{soi} and \eqref{AA}, we find a constant $c_1(\mu,\phi)>0$ such that
\begin{align}\label{EUP}\|v_t^{\vv,2}\|_\infty&\le c_1(\mu,\phi)\alpha(t^{\frac{1}{2}}) \bigg(\int_0^t\ff{ \aa(s^{\ff 1 2})^2 }{s}\d s\bigg)^{\ff 1 2},\ \ t\in [0,T], \vv\in (0,1].
\end{align}
This together with \eqref{W0} yields  that for some  constant $c_2(\mu,\phi)>0$,
\begin{align*}\|v_t^\vv\|_\infty^2 &\le 2 \|v_t^{\vv,1}\|_{\infty}^2 +2\|v_t^{\vv,2}\|_\infty^2\\
 &\le c_2(\mu,\phi)\bigg(\ff{\aa(t^{\ff 1 2 })^2}{ t} +\alpha(t^{\frac{1}{2}}) \int_0^t\ff{ \aa(s^{\ff 1 2})^2 }{s}\d s\bigg),\ \ t\in (0,T],\vv\in(0,1].
\end{align*}
By the definition of $\alpha$ and \eqref{ASS}, we find a constant $c_3(\mu,\phi)>0$ such that
\begin{align}\label{vtv}& \|v_t^\vv\|_\infty^2 \le c_3(\mu,\phi)\ff{\aa(t^{\ff 1 2 })^2}{ t},\   \ t\in (0,T],\ \ \vv\in (0,1]. \end{align}

(c)  Estimate on $\|v_t^{\vv,2}-\hat v_t^2\|_{\infty}$.
Similarly to (b), we have
 \beq\label{RT} \beg{split} &\ff{R_t^{\vv,x} -1}\vv=\int_0^t R_s^{\vv,x} \big\<\vv^{-1}\eta_s^{\vv,x},\d W_s\big\>\\
&= \int_0^t R_s^{\vv,x} \Big\<\ff{\zeta_s(X_s^{x+\vv\phi(x),\mu})[b_s^{(1)}(\cdot, P_s^*\mu_\vv)- b_s^{(1)}(\cdot, P_s^*\mu)](X_s^{x+\vv\phi(x),\mu})}\vv,\ \d W_s\Big\>\\
&= h_t(\vv,x)+\int_0^t \Big\<\zeta_s(X_s^{x,\mu})v_s^\vv(X_s^{x,\mu},P_s^\ast\mu),\ \d W_s\Big\>,\ \ x\in\R^d,\end{split}\end{equation}
where
\beg{align*} h_t(\vv,x):= \int_0^t \Big\<&\zeta_s(X_s^{x+\vv\phi(x),\mu})R_s^{\vv,x} \ff{[b_s^{(1)}(\cdot, P_s^*\mu_\vv)- b_s^{(1)}(\cdot, P_s^*\mu)](X_s^{x+\vv\phi(x),\mu})}\vv\\
&\qquad\qquad\quad-\zeta_s(X_s^{x,\mu})v_s^\vv(X_s^{x,\mu},P_s^\ast\mu),\ \d W_s\Big\>,\ \ x\in\R^d\end{align*}
satisfies
\beq\label{HH}\lim_{\vv\to 0} \E\Big[\sup_{t\in [0,T]} |h_t(\vv,x)|^2\Big] =0,\ \ x\in\R^d.\end{equation}
Indeed, by \eqref{B-B} and the definition of $v_s^\vv$, we have
\begin{align*}
 &\frac{[b_s^{(1)}(\cdot,P_s^\ast\mu_\vv)-b_s^{(1)}(\cdot,P_s^\ast\mu)]((X_s^{x+\vv\phi(x),\mu})}{\vv}\\
&=\frac{1}{\vv}\int_0^1\int_{\R^d}D^E b_s^{(1)}(X_s^{x+\vv\phi(x),\mu},(1-r)P_s^\ast\mu+rP_s^\ast\mu_\vv)(z)(P_s^\ast\mu_\vv-P_s^\ast\mu)(\d z)\d r\\
&=\int_0^1v_s^\vv(X_s^{x+\vv\phi(x),\mu},(1-r)P_s^\ast\mu+rP_s^\ast\mu_\vv)\d r.
\end{align*}
This together with the BDG inequality implies
\begin{align}\label{nyt}
\nonumber& \E\Big[\sup_{t\in [0,T]} |h_t(\vv,x)|^2\Big]\\
&\leq 2\int_0^T\E \Big|\zeta_s(X_s^{x+\vv\phi(x),\mu})R_s^{\vv,x} \int_0^1v_s^\vv(X_s^{x+\vv\phi(x),\mu},(1-r)P_s^\ast\mu+rP_s^\ast\mu_\vv)\d r\\
\nonumber&\qquad\qquad\quad-\zeta_s(X_s^{x,\mu})v_s^\vv(X_s^{x,\mu},P_s^\ast\mu)\Big|^2\d s.
\end{align}
By \eqref{EX}, for any $p>1$, we can find a constant $c_p>0$ such that
\beq\label{E3'} \E\Big[\sup_{t\in [0,T]} |X_t^{x+\vv\phi(x),\mu}-X_t^{x,\mu}|^p\Big]\le c_p|\phi(x)|^p\vv^p,\ \ \vv\in [0,1],\mu\in \scr P_k.\end{equation}
By the boundedness and continuity of $\zeta$ due to {\bf(B)}, $\int_0^T \ff{\aa(t^{\ff 1 2})^2}{t}  \d t<\infty$, \eqref{E3'}, \eqref{EE'}, \eqref{vtv}, \eqref{nyt},  and the dominated convergence theorem, to prove \eqref{HH}, it is sufficient to prove that for $(s,x,r)\in(0,T]\times\R^d\times[0,1]$,
\begin{align}\label{gyt}
\lim_{\vv\to0}\E\Big|v_s^\vv(X_s^{x+\vv\phi(x),\mu},(1-r)P_s^\ast\mu+rP_s^\ast\mu_\vv)- v_s^\vv(X_s^{x,\mu},P_s^\ast\mu)\Big|=0.
\end{align}
For any $(\omega,\omega')\in\Omega\times\Omega$, let
\begin{align*}&U^{1,\vv}_r(x,y,s,u,\omega,\omega') =U_s(X_s^{x+\vv\phi(x),\mu}(\omega'),(1-r)P_s^\ast\mu+rP_s^\ast\mu_\vv,X_s^{y+\vv u \phi(y),\mu}(\omega)),\\
&U^{2,\vv}_r(x,y,s,u,\omega,\omega')=U_s(X_s^{x+\vv\phi(x),\mu}(\omega'),(1-r)P_s^\ast\mu+rP_s^\ast\mu_\vv,  \tilde{\theta}^{\lambda,\mu}_s(y+\vv u \phi(y))),\\
&\tilde{U}^{1,\vv}_r(x,y,s,u,\omega,\omega')=U_s(X_s^{x,\mu}(\omega'),P_s^\ast\mu,X_s^{y+\vv u \phi(y),\mu}(\omega)),\\
&\tilde{U}^{2,\vv}_r(x,y,s,u,\omega,\omega')=U_s(X_s^{x,\mu}(\omega'),P_s^\ast\mu, \tilde{\theta}^{\lambda,\mu}_s(y+\vv u \phi(y))).
\end{align*}
Since {\bf (B)} implies {\bf(A)}, \eqref{END} holds such that
\begin{align}\label{puv}\W_k((1-r)P_s^\ast\mu+rP_s^\ast\mu_\vv,P_s^\ast\mu)\leq r\W_k(P_s^\ast\mu_\vv,P_s^\ast\mu)\leq cr\vv\|\phi\|_{L^k(\mu)}.\end{align}
 By \eqref{v1a}, 
\eqref{EX} and H\"{o}lder's inequality, we conclude that for any $\beta\in(1,k)$,
\begin{align*}&\E|v_s^{\vv,1}(X_s^{x+\vv\phi(x),\mu},(1-r)P_s^\ast\mu+rP_s^\ast\mu_\vv) -v_s^{\vv,1}(X_s^{x,\mu},P_s^\ast\mu)|\\
&\leq\int_{0}^1\int_{\R^d}\int_{\Omega\times \Omega}\bigg|\big[(U^{1,\vv}_r(x,y,s,u,\omega,\omega')-U^{2,\vv}_r(x,y,s,u,\omega,\omega'))\\
&\qquad\qquad\qquad\quad-(\tilde{U}^{1,\vv}_r(x,y,s,u,\omega,\omega')- \tilde{U}^{2,\vv}_r(x,y,s,u,\omega,\omega'))\big]\\
&\qquad\qquad\quad\times \int_0^s\ff 1 s \big\<\zeta_v(X_v^{y+\vv u \phi(y),\mu})\nn_{\phi(y)} X_v^{y+\vv u\phi(y),\mu},\d W_v\big\>\bigg|\d\P(\omega)\d \P(\omega')\mu(\d y)\d u \\
&\leq c_0\int_0^1 \int_{\R^d}\frac{1}{\sqrt{s}}|\phi(y)|\Big\{\int_{\Omega\times\Omega} \bigg|(U^{1,\vv}_r(x,y,s,u,\omega,\omega')-U^{2,\vv}_r(x,y,s,u,\omega,\omega'))\\
&\qquad\qquad\qquad\quad-(\tilde{U}^{1,\vv}_r(x,y,s,u,\omega,\omega')- \tilde{U}^{2,\vv}_r(x,y,s,u,\omega,\omega'))\bigg|^\beta\d \P(\omega)\d \P(\omega')\Big\}^{\frac{1}{\beta}}\mu(\d y)\d u.
\end{align*}
 {By \eqref{ale} for $\eta=\alpha(\xi)^{k-1}$ and $p=\frac{k}{k-1}$, we obtain $\|\alpha(\xi)\|_{L^k(\P)}\leq \alpha(\|\xi\|_{L^k(\P)})$, which together with \eqref{soi} implies}
\begin{align*}
&\int_{\Omega\times\Omega} \bigg|(U^{1,\vv}_r(x,y,s,u,\omega,\omega')-U^{2,\vv}_r(x,y,s,u,\omega,\omega'))\\
&\qquad\qquad\qquad\quad-(\tilde{U}^{1,\vv}_r(x,y,s,u,\omega,\omega')- \tilde{U}^{2,\vv}_r(x,y,s,u,\omega,\omega'))\bigg|^k\d \P(\omega)\d \P(\omega')\\
&\leq 2^{k}\E\alpha(|X_s^{y+\vv u \phi(y),\mu}- \tilde{\theta}^{\lambda,\mu}_s(y+\vv u \phi(y))|)^k\\
&\leq 2^{k}\alpha\left((\E|X_s^{y+\vv u \phi(y),\mu}- \tilde{\theta}^{\lambda,\mu}_s(y+\vv u \phi(y))|^k)^{\frac{1}{k}}\right)^k\\
&\leq c_k\alpha(\sqrt{s})^{k}
\end{align*}
for some constant $c_k>0$.
So, it follows from the fact that $D^E b_t^{(1)}(y,\nu)(z)$ is continuous in $(y,\nu,z)\in\R^d\times\scr P_k\times\R^d$ due to ($B_2$), \eqref{puv}, \eqref{E3'}, \eqref{VW} and the dominated convergence theorem that
\begin{align*}
\lim_{\vv\to0}\E\Big|v_s^{\vv,1}(X_s^{x+\vv\phi(x),\mu},(1-r)P_s^\ast\mu+rP_s^\ast\mu_\vv)\d r-v_s^{\vv,1}(X_s^{x,\mu},P_s^\ast\mu)\Big|=0.
\end{align*}
Similarly, by \eqref{B00} and \eqref{EE'}, we have
\begin{align*}
\lim_{\vv\to0}\E\Big|v_s^{\vv,2}(X_s^{x+\vv\phi(x),\mu},(1-r)P_s^\ast\mu+rP_s^\ast\mu_\vv)\d r-v_s^{\vv,2}(X_s^{x,\mu},P_s^\ast\mu)\Big|=0.
\end{align*}
Therefore,   \eqref{gyt} holds, which implies    \eqref{HH}   as explained before \eqref{gyt}.

Moreover, by \eqref{HH}, \eqref{W2}, \eqref{soi}, \eqref{vtv}, \eqref{nyt} and  the  argument leading to \eqref{IT1}, we obtain from the dominated convergence theorem that
\begin{align}\label{htv}
\lim_{\varepsilon\to0}\sup_{t\in (0,T]} \ff{\ss t}{\aa(t^{\ff 1 2})} \int_{\R^d}\|\E [U_t(\cdot,\cdot,X_t^{x,\mu})h_t(\varepsilon,x)]\|_{\infty}\mu(\d x)=0,
\end{align}
and
\begin{align}\label{cyt}
\nonumber&\lim_{\varepsilon\to0}\sup_{t\in (0,T]} \ff{\ss t}{\aa(t^{\ff 1 2})} \int_{\R^d} \bigg\|\E\bigg[U_t(\cdot,\cdot,X_t^{x,\mu})\\
&\qquad\qquad\quad\times \int_0^t
\Big\<\zeta_s(X_s^{x,\mu})\big\{[v^{\vv,1}_s-\hat v_s^1](X_s^{x,\mu},P_s^\ast\mu)\big\},  \ \d W_s\Big\>\bigg]\bigg\|_\infty\mu(\d x)=0.
\end{align}
Moreover, combining \eqref{B00} with  $[U_t(y,\nu,\cdot)]_\aa\le 1,$ and \eqref{ale} for $p=k^\ast$, we obtain
   \beg{align*} &\Big\|v_t^{\vv,2}- \int_{\R^d}\ff 1 \vv \E[U_t(\cdot,\cdot,X_t^{x,\mu})(R_t^{\vv,x}-1)]\mu(\d x)\Big\|_\infty\\
   &\le \ff 1\vv \int_{\R^d}\E[\aa(|X_t^{x,\mu}-X_t^{x+\vv\phi(x),\mu}| )|R_t^{\vv,x}-1| ] \mu(\d x)\\
 &\le \ff 1 \vv  \int_{\R^d} \left\{\Big(\E[|R_t^{\vv,x}-1|^{k^*}]\Big)^{\ff 1 {k^*}} \aa\Big(\big(\E[|X_t^{x,\mu}-X_t^{x+\vv\phi(x),\mu}|^k]\big)^{\ff 1 k}\Big)\right\}\mu(\d x).\end{align*}
This together with \eqref{E3'}   and \eqref{EE'} yields that for some constant $k_1(\mu,\phi)>0$,
$$ \Big\|v_t^{\vv,2}- \int_{\R^d}\ff 1 \vv \E[U_t(\cdot,\cdot,X_t^{x,\mu})(R_t^{\vv,x}-1)]\mu(\d x)\Big\|_\infty\le   k_1(\mu,\phi) {{\alpha( \vv)}},\ \ t\in [0,T],\ \vv\in (0,1]. $$
Combining this with \eqref{v2c}, \eqref{B00}, \eqref{HH}, \eqref{cyt}, \eqref{RT}, \eqref{htv},  and the   argument leading to \eqref{tid},
  we find a constant $k_2(\mu,\phi)$ and a measurable  function $\tt h:   (0,T] \times (0,1] \to (0,\infty)$  with
\begin{align}\label{hva}\sup_{\vv\in (0,1],t\in (0,T]} \ff{\ss t}{\aa(t^{\ff 1 2})} \tt h_t(\vv)\le k_2(\mu,\phi),\ \ \lim_{\vv\to 0} \sup_{t\in (0,T]} \ff{\ss t}{\aa(t^{\ff 1 2})} \tt h_t(\vv)=0
\end{align} such that
\beg{align}\label{mcy} \nonumber\left\|v_t^{\vv,2}-\hat v_t^2\right\|_\infty&\le   \tt h_t(\vv) +\int_{\R^d}\bigg\|\E\bigg[U_t(\cdot,\cdot,X_t^{x,\mu})\\
&\qquad\qquad\quad\times \int_0^t
\Big\<\{\zeta_s[v^{\vv,2}_s-\hat v_s^2]\}(X_s^{x,\mu},P_s^\ast\mu),  \ \d W_s\Big\>\bigg]\bigg\|_\infty\mu(\d x)\\
\nonumber &\le  \tt h_t(\vv) + k_2(\mu,\phi)\bigg(\int_0^t \|v_s^{\vv,2}-\hat v_s^2\|_\infty^2\d s\bigg)^{\frac{1}{2}},\ \ t\in (0,T].\end{align}

(d)  {Proof of \eqref{SFF}}. Let
\begin{equation*}\beg{split}  \beta_t:= \limsup_{\varepsilon\to0}\sup_{s\in (0,t]}\ff{\ss s}{ \aa(s^{\ff 1 2 })}\|v_s^{\vv,2}-\hat v_s^2\|_\infty.\end{split}\end{equation*}
Noting that    \eqref{EST0}, \eqref{W1'}, \eqref{v12}, \eqref{IT1} and \eqref{EUP} imply that $\beta_t$ satisfies
$$\sup_{t\in (0,T]} \beta_t\le \sup_{\vv\in (0,1]} \sup_{s\in (0,T]} \ff{\ss s}{\aa(s^{\ff 1 2})} \|v_s^{\vv,2}-\hat v_s^2\|_\infty =: \tt c(\mu,\phi)<\infty,$$
 {so that by Fatou's lemma in \eqref{mcy}   we derive from \eqref{hva} that}
$$ \beta_t^2  \le   C  k_2(\mu,\phi)^2    \int_0^t\ff{\aa(s^{\ff 1 2})^2}{s} \beta_s^2\d s,\ \ t\in (0,T],$$
where by \eqref{AA2},
$$C:= \sup_{t\in (0,T]} \ff{t}{\aa(t^{\ff 1 2})^2}<\infty.$$
 Combining this with $\int_0^T \ff{\aa(t^{\ff 1 2})^2}{t}  \d t<\infty$, and applying Gronwall's inequality, we prove   \eqref{SFF}, which together with \eqref{W2} completes the proof.
  \end{proof}

 We are now ready to prove Theorem \ref{TA2'}.

\beg{proof}[Proof of Theorem \ref{TA2'}]   By
  \eqref{BS0} and   \eqref{BS1'}, it suffices to prove that for any $t\in (0,T]$ and $f\in \B_{k-1,b}(\R^d)$,
\beq\label{ENN} \beg{split}  &\lim_{\vv\downarrow 0} \ff{P_t^{\mu_\vv} f(\mu_\vv)- P_t^{\mu} f(\mu_\vv)}\vv\\
&= \int_{\R^d}\E\bigg[f(X_t^{x,\mu})\int_0^t \Big\<\zeta_s (X_s^{x,\mu}) N_s(\mu,\phi),\ \d W_s\Big\>\bigg]\mu(\d x).\end{split}\end{equation}
Let $R_t^{\vv,x}$ be in \eqref{RTE}.  By \eqref{B00} for $f$ replacing $U_t(y,\nu,\cdot)$, we obtain
\beq\label{EN1} \ff{P_t^{\mu_\vv} f(\mu_\vv)- P_t^{\mu} f(\mu_\vv)}\vv=\int_{\R^d}\ff 1\vv \E\big[f(X_t^{x+\vv\phi(x),\mu})(R_t^{\vv,x}-1)\big]\mu(\d x),\ \ t\in (0,T].\end{equation}
 Noting that \eqref{E3'} implies
 $$\lim_{\vv\to 0}\E\Big[\sup_{t\in[0,T]} |X_t^{x+\vv\phi(x),\mu}-X_t^{x,\mu}|^k\Big]=0,$$
 while \eqref{HH}, \eqref{RT}, Lemma \ref{LL2}, \eqref{P00}, \eqref{W1'}, \eqref{00} and \eqref{SFF} lead to
\begin{align*}\lim_{\vv\to 0}\ff{R_t^{\vv,x}-1}\vv =  \int_0^t \Big\<\zeta_s (X_s^{x,\mu})N_s(\mu,\phi),\ \d W_s\Big\>
 \end{align*}
 in $L^2(\P)$, by taking $\vv\to 0$ in \eqref{EN1} and using the dominated convergence theorem, we deduce \eqref{ENN}  for $f\in C_b(\R^d)$. By an approximation argument as in \cite[Proof of (2.3)]{FYW3} for $f\in\scr B_b(\R^d)$, this implies  \eqref{ENN} for $f\in \B_b(\R^d).$ By the approximation argument used in the proof of \eqref{BS00}, we may further extend \eqref{ENN} to $f\in \B_{k-1,b}(\R^d).$
\end{proof}

\end{document}